\documentclass[a4paper,10pt]{article}

\usepackage{amssymb} 
\usepackage{amsfonts}
\usepackage{amsmath}
\usepackage{amsthm}
\usepackage[utf8]{inputenc}
\usepackage{pgf}%,pgfarrows}
\usepackage{graphicx}
\usepackage{epsfig}
\usepackage{epstopdf}
\usepackage{tikz}
\usepackage{relsize,exscale}
\usepackage{pdfpages}
\usepackage[all]{xy}
\usetikzlibrary{patterns}
\title{{\bfseries Real algebraic surfaces with many handles in $(\CP^1)^3$}}
\author{\textbf{Arthur Renaudineau}}
\date{}

\newtheorem{theorem}{Theorem}
\newtheorem{lemma}{Lemma}
\newtheorem{remark}{Remark}
\newtheorem{corollary}{Corollary}
\newtheorem{example}{Example}
\newtheorem{definition}{Definition}
\newtheorem*{conjecture}{Conjecture}
\newtheorem{proposition}{Proposition}

\newcommand{\R}{\mathbb{R}}
\newcommand{\RP}{\mathbb{RP}}
\newcommand{\Z}{\mathbb{Z}}
\newcommand{\CP}{\mathbb{CP}}

\newcommand{\PP}{\mathbb{P}}
\newcommand{\C}{\mathbb{C}}
\newcommand{\CC}{\mathbb{C}}

\newcommand{\SSS}{\mathbb{S}}
\begin{document}
\maketitle
\begin{abstract}
In this text, we study Viro's conjecture and related problems for real algebraic surfaces in $(\CP^1)^3$. We construct a counter-example to Viro's conjecture in tridegree $(4,4,2)$ and a family of real algebraic surfaces of tridegree $(2k,2l,2)$ in $(\CP^1)^3$ with asymptotically maximal first Betti number of the real part. To perform such constructions, we consider double covers of blow-ups of $(\CP^1)^2$ and we glue singular curves with special position of the singularities adapting the proof of Shustin's theorem for gluing singular hypersurfaces.
\end{abstract}

\section{Introduction}
A \textit{real algebraic variety} is a complex algebraic variety $X$ equipped with an antiholomorphic involution $c:X\rightarrow X$. Such an antiholomorphic involution is called a \textit{real structure} on $X$. The \textit{real part} of $(X,c)$, denoted by $\R X$, is the set of points fixed by $c$. Unless otherwise specified, all varieties considered are nonsingular. For a topological space $A$, we put $b_i(A)=\dim_{\Z/2\Z}H_i(A,\Z/2\Z)$. The numbers $b_i(A)$ are called \textit{Betti numbers $($\textit{with $\Z/2\Z$ coefficients}$)$ of $A$}.
\\Let $X$ be a simply-connected projective real algebraic surface. From the Smith-Thom inequality and the Comessatti inequalities (see for example ~\cite{DegKhar}), one can deduce the following bounds for $b_0(\R X)$ and $b_1(\R X)$ in terms of Hodge numbers of $X$:
\begin{equation}
b_0(\R X)\leq \frac{1}{2}(h^{2,0}(X)+h^{1,1}(X)+1),
\label{equation1}
\end{equation}
\begin{equation}
b_1(\R X)\leq h^{2,0}(X)+h^{1,1}(X).
\label{equation2}
\end{equation}
In 1980, O. Viro formulated the following conjecture.
\begin{conjecture}$($O. Viro$)$
\\Let $X$ be a simply-connected projective real algebraic surface. Then
$$
b_1(\R X)\leq h^{1,1}(X).
$$
\end{conjecture}
When $X$ is the double covering of $\CP^2$ ramified along a curve of an even degree, this conjecture is a reformulation of Ragsdale's conjecture (see ~\cite{Viro80}). The first counterexample to Ragsdale's conjecture was constructed by I. Itenberg (see ~\cite{Itenberg93}) using Viro's combinatorial patchworking (see for example ~\cite{Itenberg1997}). This first counterexample opened the way to various counterexamples to Viro's conjecture and constructions of real algebraic surfaces with many connected components (see ~\cite{ItKhar}, ~\cite{Itenberg1997}, ~\cite{Bihan99}, ~\cite{Bihan2001}, ~\cite{Orevkov2001} ~\cite{Brugalle2006} and ~\cite{Renau}). It is not known whether Viro's conjecture is true for $M$-surfaces (a real algebraic variety $X$ is called an \textit{$M$-variety} if $\sum b_i(\R X)= \sum b_i(X)$ and an \textit{$(M-a)$-variety} if $\sum b_i(\R X)= \sum b_i(X)-2a$).

In this text, we focus on real algebraic surfaces in $(\CP^1)^3$ equipped with the real structure given by the standard conjugation on each factor. A real algebraic surface $X$ in $(\CP^1)^3$ of tridegree $(d_1,d_2,d_3)$ is the zero set of a real polynomial 
$$ 
P\in\R\left[u_1,v_1,u_2,v_2,u_3,v_3\right]
$$ 
homogeneous of degree $d_i$ in the variables $(u_i,v_i)$, for $1\leq i\leq 3$. Up to change of coordinates, one can always assume that $d_1\geq d_2\geq d_3$. Introduce the projection \linebreak $\pi:(\CP^1)^3\rightarrow (\CP^1)^2$ on the first two factors. If $X$ is an algebraic surface of tridegree $(d_1,d_2,1)$ in $(\CP^1)^3$ , then $\pi\vert_X$ is of degree $1$, and $X$ is birationaly equivalent to $(\CP^1)^2$. Hence $h^{2,0}(X)=0$, so Viro's conjecture is true for real algebraic surfaces of tridegree $(d_1,d_2,1)$ in $(\CP^1)^3$. Assume now that $X$ is of tridegree $(d,2,2)$. The projection $\tilde{\pi}: (\CP^1)^3\rightarrow \CP^1$ on the first factor induces an elliptic fibration on $X$. Kharlamov proved (see \cite{Ak-Ma}) that Viro's conjecture is true for elliptic surfaces. Thus, Viro's conjecture is true for real algebraic surfaces of tridegree  $(d,2,2)$ in $(\CP^1)^3$. 
\\Let $X$ be a real algebraic surface of tridegree $(4,4,2)$ in $(\CP^1)^3$. One has $h^{2,0}(X)=9$ and $h^{1,1}(X)=84$. Therefore, using Smith-Thom inequality and one of the two Comessatti inequalities, one obtains
$$
b_1(\R X)\leq 92= h^{1,1}(X)+8.
$$ 
We prove the following result in Section \ref{construction2}.
\begin{theorem}
\label{theorembis}
There exists a real algebraic surface $X$ of tridegree $(4,4,2)$ in $(\CP^1)^3$ such that
$$
\R X\simeq 3S\coprod 2S_2\coprod  S_{40}.
$$
The surface $X$ is an $(M-2)$-surface satisfying 
$$
b_1(\R X)=88=h^{1,1}(X)+4.
$$ 
\end{theorem}
The existence of a real algebraic surface of tridegree $(4,4,2)$ in $(\CP^1)^3$ satisfying $90 \leq b_1(\R X) \leq 92$ is still unknown. For $d\geq 3$, the existence of a real algebraic surface of tridegree $(d,3,2)$ in $(\CP^{1})^{3}$ disproving Viro's conjecture is also unknown. 
In Section \ref{constructionasympt}, we focus on the asymptotic behaviour of the first Betti number for real algebraic surfaces of tridegree $(d_1,d_2,2)$ in $(\CP^{1})^{3}$. Let $X$ be a real algebraic surface of tridegree $(d_1,d_2,2)$ in $(\CP^1)^3$. One has 
$$
h^{2,0}(X)=d_1d_2-d_1-d_2+1,
$$
and 
$$
h^{1,1}(X)=6d_1d_2-2d_1-2d_2+4.
$$
If $\mathcal{S}_{d_1,d_2}$ denotes the set of nonsingular real algebraic surfaces of tridegree $(d_1,d_2,2)$ in $(\CP^{1})^{3}$, then it follows from Smith-Thom inequality and one of the two Comessatti inequalities that
$$
 \max_{X\in\mathcal{S}_{d_1,d_2}} b_1(\R X) \leq 7d_1d_2 -3d_1-3d_2+5.
$$
We prove the following result in Section \ref{constructionasympt}.
\begin{theorem}
\label{theorem1}
There exists a family $(X_{k,l})$ of nonsingular real algebraic surfaces of tridegree $(2k,2l,2)$ in $(\CP^1)^3$, and $A,B,c,d,e\in\Z$ such that for all $k\geq A$ and for all $l\geq B$, one has
$$
b_1(\R X_{k,l})\geq 7\cdot 2k\cdot 2l-c \cdot 2k-d \cdot 2l+e.
$$
\end{theorem}
In order to prove Theorem \ref{theorembis} and Theorem \ref{theorem1}, we consider double coverings of blow-ups of $(\CP^1)^2$ and we adapt the proof of Shustin's theorem for gluing singular hypersurfaces.
%we consider the following method of construction. We first show that a surface of tridegree $(2k,2l,2)$ in $(\CP^1)^3$ can be obtained as a small deformation of a ramified double covering of the blow-up of $(\CP^1)^2$ along $p_1,\cdots, p_{2kl}$, the $2kl$ intersection points of two real algebraic curves of bidegree $(k,l)$ in $(\CP^1)^2$. The ramification locus is the strict transform of a real algebraic curve of bidegree $(4k,4l)$ in $(\CP^1)^2$ with a double point on each $p_i$, for $1\leq i \leq 2kl$. In order to construct such curves, we glue some collection of triples $\left(C_i,L_i,M_i\right)$, where $C_i$ is a curve with some  double points located on the intersection $L_i\cap M_i$. In order to perform such gluing, we adapt the proof of Shustin's patchwork for singular curves (see \cite{Shustin}).
This text is organized as follows. In Section \ref{singulargluing}, we recall briefly Shustin's theorem for gluing singular hypersurfaces. In Section \ref{lowdegree}, we discuss orientability of closed two-dimensional submanifolds of $(\RP^{1})^{3}$. In Section \ref{doublecovering}, we present a way to construct a surface of degree $(2k,2l,2)$ in $(\CP^1)^3$ as a small perturbation of a double covering of a blow-up of $(\CP^1)^2$. In Section \ref{constructionasympt}, we prove Theorem \ref{theorem1} and in Section \ref{construction2} we prove Theorem \ref{theorembis}. In Section \ref{Transversality}, we recall Brusotti theorem and prove a transversality theorem needed in our construction. 

\textbf{Acknowledgments.} I am very grateful to Erwan Brugallé and Ilia Itenberg for very useful discussions. I am also grateful to Eugenii Shustin for very useful remarks and for answering all my questions.

\section{Viro's method for gluing singular hypersurfaces}
\label{singulargluing}
Here we follow \cite{Shustin}. Another reference is the book \cite{IMS}. Let us first fix the notations and introduce standard definitions.
\begin{definition}
Let $f=\sum a_i z^i$ be a polynomial in $\CC\left[z_1,\cdots ,z_n\right]$. The convex hull of the set
$
\left\lbrace i\in\Z^n \mid a_i\neq 0 \right\rbrace
$
is called the Newton polytope of $f$. An integer convex polytope in $\R^{n}$ is the convex hull of a finite subset of $\Z^{n}\subset\R^n$. For an integer convex polytope $\Delta$, denote by $\mathcal{P}\left(\Delta\right)$ the space of polynomials with Newton polytope $\Delta$. 
\end{definition}
\begin{definition}
Let $f=\sum a_iz^i\in\mathcal{P}(\Delta)$.
Let $\Gamma\subset\Z^n$ be a subset of $\Delta$. The \textit{truncation} of $f$ to $\Gamma$ is the polynomial $f^\Gamma$ defined by 
$
f^\Gamma=\sum_{i\in\Gamma} a_iz^i.
$
\end{definition}
\begin{definition}
A subdivision of an integer convex polytope $\Delta$ is a set of integer convex polytopes $(\Delta_i)_{i\in I}$ such that:

\begin{itemize}
\item $\cup_{i\in I}\Delta_i=\Delta$,
\item if $i,j\in I$, then the intersection $\Delta_i\cap\Delta_j$ is a common face of the polytope $\Delta_i$ and the polytope $\Delta_j$, or empty.
\end{itemize}
A subdivision $\cup_{i\in I}\Delta_i$ of $\Delta$ is said to be \textit{convex} if there exists a convex piecewise-linear function $\nu:\Delta\rightarrow\R$ whose domains of linearity coincide with the polytopes $\Delta_i$.
\end{definition}

\begin{definition}
\label{definitionchart}
Let $\Delta\subset(\R_+)^n$ be an integer convex n-dimensional polytope. Let $f\in\mathcal{P}(\Delta)$, and let $Z(f)$ be the set $\lbrace x\in(\R^*)^n\vert\mbox{ } f(x)=0\rbrace.$
In the orthant $(\R_+^*)^n$, define $\phi$ as %the restriction of $g\circ \psi$ to $(\R_+)^3$ where
$$
\begin{array}{ccccc}
\phi & : & (\R_+^*)^n & \to & (\R^*)^n\\
 & & x & \mapsto & \dfrac{\sum_{i\in\Delta\cap\Z^n}\mid x^{i}\mid i}{\sum_{i\in\Delta\cap\Z^n}\mid x^{i}\mid}.  \\
\end{array}
$$
%and
%$$
%\begin{array}{ccccc}
%g & : & \CP^s & \to & (\R^*)^3\\
% & & \left[y_0:...:y_s\right] & \mapsto & \dfrac{\sum_{0\leq i\leq s}\mid z^{w_i}\mid w_i}{\sum_{0\leq i\leq s}\mid z^{w_i}\mid}.  \\
%\end{array}
%$$
In the orthant $s_\varepsilon((\R_+^*)^n)$, we put 
$$
\phi(s_\varepsilon(x))=s_\varepsilon(\phi(x)),
$$
where $s_\varepsilon(x_1,\cdots ,x_n)=((-1)^{\varepsilon_1} x_1,\cdots ,(-1)^{\varepsilon_n} x_n)$.
\\We call chart of $f$ the closure of $\phi(Z(f))$ in $\Delta_*$, where
$$
\Delta_{*}=\bigcup_{\varepsilon\in(\Z/2\Z)^n}s_\varepsilon(\Delta).
$$ Denote by $C(f)$ the chart of $f$.
\end{definition}
\begin{definition}
Le $\Delta$ be an integer convex polytope in $\R^{n}$. If $\Gamma$ is a face of $\Delta_*$, then, for all integer vectors $\alpha$ orthogonal to $\Gamma$ and for all $x\in\Gamma$, identify $x$ with $s_\alpha(x)$. Denote by $\hat{\Delta}$ the quotient of $\Delta_*$ under these identifications, and denote by $\pi_\Delta$ the quotient map.
\end{definition}

The term \textit{singular point} of a polynomial $f$ in $n$ variables means a singular point of the hypersurface $\lbrace f=0\rbrace\cap (\CC^*)^n$. We study real polynomials with only finitely many singular points.  Denote by $\mathrm{Sing}(f)$ the set of singularities of $f$.
Let us be given a certain classification $\mathcal{S}$ of isolated hypersurface singularities which are invariant with respect to the transformations
$$
f(z_1,\cdots ,z_n)\rightarrow \lambda_0 f(\lambda_1 z_1,\cdots ,\lambda_n z_n), 
$$
where $\lambda_0,\cdots ,\lambda_n > 0$. In addition, assume that in each type of the classification, the Milnor number is constant. For a polynomial $f$, denote by $\mathcal{S}(f)$ the function 
$$
\begin{array}{ccccc}
\mathcal{S}(f) & : & \mathcal{S} & \longrightarrow & \Z\\
 & & s & \mapsto & \#\left\lbrace z\in (\CC^{*})^{n}\mid z \mbox{ is in } \mathrm{Sing}(f) \mbox{ of type } s\right\rbrace.\\
\end{array}
$$
\begin{definition}
A polynomial $f\in\mathcal{P}\left(\Delta\right)$ is called peripherally nonsingular $($PNS$)$ if for every proper face $\Gamma\subset\Delta$, the hypersurface $Z(f^\Gamma)=\lbrace z\in(\CC^*)^n\mid f^\Gamma(z)=0\rbrace$ is nonsingular.
\end{definition}
\begin{definition}	
Let $f\in\mathcal{P}\left(\Delta\right)$ and $\partial\Delta_+\subset \partial\Delta$ be the union of some facets of $\Delta$. Put
$$
\mathcal{P}\left(\Delta,\partial\Delta_+,f\right)=\left\lbrace g\in\mathcal{P}\left(\Delta\right) \mid g^\Gamma=f^\Gamma \:\: \mbox{ for any facet } \:\: \Gamma\subset\partial\Delta_+\right\rbrace.
$$
\end{definition}
\begin{definition}
\label{S-transversality}
Let $f$ be a polynomial with only finitely many singular points and let $v_1, \cdots , v_m$ be all the singular points of $f$. In the space $\Sigma(d)$ of polynomials of degree less than $d$, for $d\geq \deg f$, consider a germ $M_d(f)$ at $f\in\Sigma(d)$ of the variety of polynomials with singular points in neighborhoods of the points $v_1,\cdots ,v_m$  of the same types. The triad $\left(\Delta, \partial\Delta_+, f\right)$ is said to be $S$-transversal if
\begin{itemize}
\item for $d\geq d_0$, the germ $M_d(f)$ is smooth and its codimension in $\Sigma(d)$ does not depend on $d$. 
\item the intersection of $M_d(f)$ and $\mathcal{P}\left(\Delta,\partial\Delta_+,f\right)$ in $\Sigma(d)$ is transversal for $d\geq d_0$.
\end{itemize}
\end{definition}
\begin{remark}
Instead of types of isolated singular points one can consider other properties of polynomials which can be localized, are invariant under the torus action, and the corresponding strata are smooth. Then one can speak of the respective transversality in the sense of Definition \ref{S-transversality}, and prove a patchworking theorem similar to that discussed below. %We will apply this method in Chapter \ref{rasinP1}. 

\end{remark}
\subsection{General gluing theorem}
Let $\Delta$ be an integer convex $n$-dimensionnal polytope in $\R_+^n$, and let $\cup_{i\in I}\Delta_i$ be a subdivision of $\Delta$. For any $i\in I$, take a polynomial $f_i$ such that the polynomials $f_i$ verify the following properties:
\begin{itemize}
\item for each $i\in I$, one has $f_i\in\mathcal{P}(\Delta_i)$,
\item if $\Gamma=\Delta_i\cap\Delta_j$, then $f_i^\Gamma=f_j^\Gamma$,
\item for each $i\in I$, the polynomial $f_i$ is PNS.
\end{itemize}
The polynomials $f_i$ define an unique polynomial $f=\sum_{w\in\Delta\cap\Z^n}a_wx^w$, such that $f^{\Delta_i}=f_i$ for all $i\in I$. Let $G$ be the adjacency graph of the subdivision $\cup_{i\in I}\Delta_i$. Define $\mathcal{G}$ to be the set of oriented graphs $\Gamma$ with support $G$ and without oriented cycles. For $\Gamma\in\mathcal{G}$, denote by $\partial\Delta_{i,+}$ the union of facets of $\Delta_i$, which correspond to the arcs of $\Gamma$ coming in $\Delta_i$.
\begin{theorem}[Shustin, see \cite{Shustin}]
\label{patchworksingular}
Assume that the subdivision $\cup_{i\in I}\Delta_i$ of $\Delta$  is convex and that there exists $\Gamma\in\mathcal{G}$ such that the triad $\left(\Delta_i,\partial\Delta_{i,+}, f_i\right)$ is S-transversal. Then, there exists a PNS polynomial $f\in\mathcal{P}(\Delta)$ such that
$$
\mathcal{S}(f)=\sum_{i\in I} \mathcal{S}(f_i),
$$
and the triad $\left(\Delta,\emptyset, f\right)$ is S-transversal.
\end{theorem}
\begin{remark}
Let $\nu:\Delta\rightarrow\R$ be a function certifying the convexity of the subdivision of $\Delta$. Then, the polynomial $f$ can be chosen of the form $f_{t_0}$, for $t_0$ positive and small enough, where
$$
f_t=\sum_{i\in\Delta\cap\Z^{n}}A_i(t)t^{\nu(i)}z^{i},
$$
and $\vert A_i(t)-a_i\vert \leq Kt$, for a positive constant $K$. The polynomial $f_t$ is a modified version of Viro polynomial. Suppose that the polynomials $f_i$ are real. Then, the polynomial $f_t$ can also be chosen real and as in Viro's theorem, for $t$ positive and small enough, the pairs $\left( \hat{\Delta},\pi_{\Delta}(C(f_t))\right)$ and $\left( \hat{\Delta},\pi_{\Delta}(\bigcup_{i\in I} C(f_i))\right)$ are homeomorphic.
\end{remark}

\subsection{S-transversality criterion}
Let $n=2$, i.e., polynomials $f_i$ define curves in toric surfaces.
In \cite{Shustin}, Shustin defined a non-negative integer topological invariant $b(w)$ of isolated planar curve singular points $w$ such that, if $f_i$ is irreducible and
$$
\sum_{w\in \mathrm{Sing}(f_i)}b(w)<\sum_{\sigma\not\subset\partial\Delta_{i,+}} \mathrm{length}(\sigma),
$$  
then the triple $\left(\Delta_i,\partial\Delta_{i,+}, f_i\right)$ is S-transversal. Here, $\mathrm{length}(\sigma)$ denotes the integer length of $\sigma$. Recall that
\begin{itemize}
\item if $w$ is a node, then $b(w)=0$,
\item if $w$ is a cusp, then $b(w)=1$,
%\item if the singularity of $w$ is defined to have tangency of order $m$ to a given line $L$ at $w$, then $b(w)=m$.
\end{itemize}
\begin{example}
If under the hypotheses of Theorem \ref{patchworksingular}, $n=2$, the curves 
$$\overline{\lbrace f_i=0\rbrace}\subset Tor(\Delta_i),$$ $i\in I$, are irreducible and have only ordinary nodes as singularities, then there is an oriented graph $\Gamma\in\mathcal{G}$ such that all the triples $\left(\Delta_i,\partial\Delta_{i,+}, f_i\right)$ are transversal. Indeed, for any common edge $\sigma\subset\Delta_i\cap\Delta_j$, one can choose the corresponding arc of $\Gamma$ to be orthogonal to $\sigma$. Then orient the arcs of $\Gamma$ so that they form angles in the interval $\left]-\frac{\pi}{2},\frac{\pi}{2}\right]$ with the horizontal axis. Then
$$
\partial\Delta_{i,+}\neq \partial\Delta_i,
$$
for $i\in I$. The above criterion implies transversality.
\end{example}

\section{Orientability of closed two-dimensional submanifolds of $(\RP^1)^3$}
\label{lowdegree}
Identify $H_2\left((\RP^1)^3 \: ; \: \Z/2\Z\right)$ with $(\Z/2\Z)^3$ by considering generators $x_1,x_2,x_3$ given by 
\begin{itemize}
\item $x_1=\left[\lbrace p \rbrace\times \RP^1\times \RP^1\right]$,
\item $x_2=\left[\RP^1\times \lbrace p \rbrace \times \RP^1\right]$,
\item $x_3=\left[\RP^1\times \RP^1\times \lbrace p \rbrace \right]$.
\end{itemize}

Identify $H_1\left((\RP^1)^3 \: ; \: \Z/2\Z \right)$ with $(\Z/2\Z)^3$ by considering generators $y_1,y_2,y_3$ given by 
\begin{itemize}
\item $y_1=\left[\RP^1\times \lbrace p \rbrace \times \lbrace q \rbrace \right]$,
\item $y_2=\left[\lbrace p \rbrace \times \RP^1 \times \lbrace q \rbrace\right]$,
\item $y_3=\left[\lbrace p \rbrace \times \lbrace q \rbrace \times \RP^1 \right]$.
\end{itemize}   
With these identifications, the intersection product $$
H_2\left((\RP^1)^3 \: ; \: \Z/2\Z \right)\times H_1\left((\RP^1)^3\: ; \: \Z/2\Z \right)\rightarrow \Z/2\Z 
$$ is given by the following map:
$$
\begin{array}{ccccc}
Q & : & (\Z/2\Z)^3\times (\Z/2\Z)^3 & \to & \Z/2\Z \\
 & & \left((a_1,a_2,a_3),(b_1,b_2,b_3)\right) & \mapsto & a_1b_1+a_2b_2+a_3b_3.  \\
\end{array}
$$
If $X$ is a real algebraic surface of tridegree $(d_1,d_2,d_3)$ in $(\CP^1)^3$, it follows from the identifications made above that the homology class of $\R X$ is represented in $(\Z/2\Z)^3$ by 
$$
\left[\R X\right] = \left(d_1\mod 2\: , d_2\mod 2 \: , d_3\mod 2\right).
$$ 
Consider the case of a real algebraic surface $X$ of tridegree $(d_1,d_2,1)$ in $(\CP^1)^3$. Such a surface is given by a polynomial 
$$
u_3 P(u_1,v_1,u_2,v_2)+v_3Q(u_1,v_1,u_2,v_2),
$$ where $P$ and $Q$ are homogeneous polynomials of degree $d_i$ in the variables $(u_i,v_i)$, for $1\leq i\leq 2$. Assume that $\left\lbrace P=0 \right\rbrace$ and $\left\lbrace Q=0 \right\rbrace$ intersect transversely. The projection $\pi\vert_X$ on the two first factors identify $X$ with the blow up of $(\CP^1)^2$ at the $2d_1d_2$ intersection points of $\lbrace P=0 \rbrace$ and $\lbrace Q=0\rbrace $. If $2m$ points of intersections of $\lbrace P=0 \rbrace$ and $\lbrace Q=0 \rbrace$ are real, then 
$$
\R X\simeq (\RP^1\times\RP^1)\#_{2m}\RP^2,
$$ 
where $\#$ denotes the connected sum. In particular, the surface $\R X$ is orientable if and only if $m=0$.
\begin{proposition}
\label{Eulerchareven}
The Euler characteristic of any closed two-dimensional submanifold $Z$ of $(\RP^1)^3$ is even.
\end{proposition}
\begin{proof}
If $[Z]=0\in H_2\left((\RP^1)^3 \: ; \: \Z/2\Z \right)$, then $Z$ is the boundary of a compact $3$-manifold $M$ in $(\RP^1)^3$. Since $(\RP^1)^3$ is orientable, the $3$-manifold $M$ is orientable and $Z$ is also orientable. It follows that the Euler characteristic of $Z$ is even. Assume now that $[Z]\neq 0\in H_2\left((\RP^1)^3 \: ; \: \Z/2\Z \right)$.
Up to a change of coordinates on $(\RP^1)^3$, one can assume that $[Z]=(1,a,b)$, with $a,b\in \Z/2\Z$. Consider a real algebraic surface $S$ of tridegree $(1,a,b)$ transverse to $Z$. Then, one has $[Z\cup S]=[Z]+[S]=0$, and the union $Z\cup S$ bounds in $(\RP^{1})^{3}$. Thus, one can color the complement $(\RP^{1})^3\setminus (Z\cup S)$ into two colors in such a way that the components adjacent from the different sides to the same (two-dimensional) piece of $Z\cup S$ would be of different colors. It is a kind of checkerboard coloring. Consider the disjoint sum $Q$ of the closures of those components of $(\RP^{1})^{3}\setminus (Z\cup S)$ which are colored with the same color. It is a compact $3$-manifold, and it is oriented since each of the components inherits orientation from $(\RP^{1})^{3}$. The boundary of this $3$-manifold is composed of pieces of $Z$ and $S$. It can be thought of as the result of cutting both surfaces along their intersection curve and regluing. The intersection curve is replaced by its two copies, while the rest part of $Z$ and $S$ does not change. Since the intersection curve consists of circles, its Euler characteristic is zero. Thus, one has 
$$
\chi(\partial Q)=\chi(Z)+\chi(S).
$$
Since the surface $S$ is the connected sum of an even number of copies of $\RP^2$, the Euler characteristic $\chi(S)$ of $S$ is even. On the other hand, $\chi(\partial Q)$ is even since $\partial Q$ inherits orientation from $Q$. Thus, one has
$$
\chi(Z)=0\mod 2.
$$ 
\end{proof}
Proposition \ref{Eulerchareven} implies that we cannot embed any connected sum of an odd number of copies of $\RP^2$ in $(\RP^1)^3$.
%Consider $Y_{k,l,m}$ the surface of degree $(2k,2l,2m)$ given by $\lbrace u_1^{2k}+v_1^{2k}+u_2^{2l}+v_2^{2l}+u_3^{2m}+v_3^{2m}=0 \rbrace$. The real part $\R Y_{k,l,m}$ is empty. Perturbing the union of $X$ and $Y_{k,l,m}$, one obtains examples of real algebraic surfaces of degree $(d_1+2k,d_2+2l,2m+1)$ with orientable and nonorientable real parts.
One has the following characterisation of orientability.
\newpage
\begin{proposition}
\label{proporientability}
Let $Z$ be a two-dimensional submanifold of $(\RP^1)^3$. The manifold $Z$ is nonorientable if and only if there exists a circle $\SSS^1\hookrightarrow Z$ such that $[\SSS^1]\cdot [Z]=1$, where 
$$
[\SSS^1]\in H_1\left((\RP^1)^3\: ; \: \Z/2\Z \right)
$$ is the homology class of the circle $\SSS^1$ and 
$$
[Z]\in H_2\left((\RP^1)^3\: ; \: \Z/2\Z \right)
$$ is the homology class of the manifold $Z$.
\end{proposition}
\begin{proof}
The manifold $Z$ is nonorientable if and only if there exists a disorienting circle $\SSS^{1} \hookrightarrow Z$. Recall that $\SSS^{1}\subset Z$ is a disorienting circle if and only if the normal bundle $\mathcal{N}_{\SSS^{1},Z}$ of $\SSS^{1}$ in $Z$ is nontrivial. On the other hand, one has the following equality:
$$
\mathcal{N}_{\SSS^{1},(\RP^1)^3}=\mathcal{N}_{\SSS^{1},Z}\oplus\left.\left(\mathcal{N}_{Z,(\RP^1)^3}\right)\right\vert_{\SSS^{1}},
$$ 
where $\mathcal{N}_{\SSS^{1},(\RP^1)^3}$ denotes the normal bundle of $\SSS^{1}$ in $(\RP^1)^3$ and $\left.\left(\mathcal{N}_{Z,(\RP^1)^3}\right)\right\vert_{\SSS^{1}}$ denotes the restriction to $\SSS^{1}$ of the normal bundle of $Z$ in $(\RP^1)^3$. The restriction of the tangent bundle of $(\RP^1)^3$ to $\SSS^{1}$ is the Whitney sum of the normal and the tangent bundle of $\SSS^1$. Thus, the normal bundle  
$\mathcal{N}_{\SSS^{1},(\RP^1)^3}$ is trivial, so the normal bundle $\mathcal{N}_{\SSS^{1},Z}$ is nontrivial if an only if the normal bundle $\left.\left(\mathcal{N}_{Z,(\RP^1)^3}\right)\right\vert_{\SSS^{1}}$ is nontrivial. One concludes that $Z$ is nonorientable if and only if there exists an embedding $\SSS^{1} \hookrightarrow Z $ such that $[\SSS^{1}]\cdot [Z]=1$. 
\end{proof}
\begin{corollary}
Let $X$ be a nonsingular real algebraic surface of degree $(k,l,2)$ in $(\CP^1)^3$ given by a polynomial
$$
u_3^2 P+u_3v_3 Q+v_3^2 R,
$$ where $P$, $Q$ and $R$ are homogeneous polynomials of degree $k$ in $(u_1,v_1)$ and of degree $l$ in $(u_2,v_2)$. Then, $\R X$ is nonorientable if an only if there exists an embedding of a circle $\psi:\SSS^1\hookrightarrow (\RP^1)^2$ such that
\begin{itemize}
\item $\psi(\SSS^1)\subset \lbrace Q^2-4PR\geq 0\rbrace$.
\item The homology class $[\psi(\SSS^1)]\in H_1\left((\RP^1)^2 \: ; \: \Z/2\Z\right)$ is represented by $(a,b)\in\Z/2\Z$ with $ak+bl=1$.
\end{itemize}
\end{corollary}
\begin{proof}
Notice that $\R X$ is homeomorphic to the disjoint union of two copies of \linebreak $\lbrace Q^2-4PR\geq 0\rbrace$ attached to each other by a self-homeomorphism of $\lbrace Q^2-4PR = 0\rbrace$. Assume that the manifold $\R X$ is nonorientable. Then, by Proposition \ref{proporientability}, there exists an embedding of a circle $\phi:\SSS^1\hookrightarrow \R X$ such that the homology class 
$$
[\phi(\SSS^1)]\in H_1\left((\RP^1)^3\: ; \: \Z/2\Z \right)
$$ is represented by $(a,b,c)\in (\Z/2\Z)^3$ with $ak+bl=1$. Using a small perturbation of the embedding $\phi$, one obtains an immersion $\pi\circ\phi$ of $\SSS^1$ in $(\RP^1)^2$, where $\pi:(\RP^1)^3\rightarrow (\RP^1)^2$ denotes the projection on the two first factors. This immersion satisfies
\begin{itemize}
\item $\pi\circ\phi(\SSS^1)\subset \lbrace Q^2-4PR\geq 0\rbrace$ and
\item $[\pi\circ\phi(\SSS^1)]=(a,b)$, with $ak+bl=1$.
\end{itemize}
Perturbing arbitrarily each double point of $\pi\circ\phi(\SSS^1)$, one obtains a collection of circles $\SSS^1_1,\cdots , \SSS^1_n$, with $\left[\SSS^1_j\right]=(a_j,b_j)\in(\Z/2\Z)^2$, such that
\begin{itemize}
\item $\SSS^1_j\subset \lbrace Q^2-4PR\geq 0\rbrace$, for all $1\leq j\leq n$,
\item $\sum_{j=1}^n a_j=a$ and $\sum_{j=1}^n b_j=b$.
\end{itemize} 
Then there exists $1\leq j\leq n$ such that $a_jk+b_jl=1$. 
\\ Reciprocally, assume that there exists an embedding $\psi:\SSS^1\hookrightarrow (\RP^1)^2$ such that
\begin{itemize}
\item $\psi(\SSS^1)\subset \lbrace Q^2-4PR\geq 0\rbrace$,
\item $[\psi(\SSS^1)]=(a,b)\in(\Z/2\Z)^2$ with $ak+bl=1$.
\end{itemize}
Using a small perturbation of the embedding $\psi$, one can assume that 
$$
\psi(\SSS^1)\subset \lbrace Q^2-4PR > 0\rbrace,
$$ 
and that $\psi(\SSS^1)$ intersect $\lbrace P=0 \rbrace$ transversely.
Over the set $\lbrace Q^2-4PR > 0\rbrace$, the map $\pi\vert_{\R X}$ is a two-to-one map. Then, the preimage of $\psi(\SSS^1)$ under the projection $\pi\vert_{\R X}$ is either a circle or a union of two circles. Consider 
$$
\R X\setminus\lbrace P=0 \rbrace \subset (\RP^1)^2\times \R.
$$ 
If $x\in\psi(\SSS^1)\setminus \lbrace P=0 \rbrace$, one can order the two preimages $s_1(x)$ and $s_2(x)$ of $x$ under $\pi$ so that the vertical coordinate of $s_1(x)$ is smaller than the vertical coordinate of $s_2(x)$. One can see that when $x$ goes through $\lbrace P=0 \rbrace$, the order on $s_1(x)$ and $s_2(x)$ changes.
Then, the preimage of $\psi(\SSS^1)$ under the projection $\pi\vert_{\R X}$ is a circle if and only if the number of intersections of $\psi(\SSS^1)$ with $\lbrace P=0 \rbrace$ is odd. Since $ak+bl=1$, the preimage of $\psi(\SSS^1)$ is a circle, and it follows from Proposition \ref{proporientability} that $\R X$ is nonorientable.
\end{proof}
%In particular, one see that Viro's conjecture is true for surfaces of degree $(d_1,d_2,1)$. Assume that the degree of $X$ is $(d_1,2,2)$. Then the projection $\pi^1:(\CP^1)^3\rightarrow\CP^1$ forgetting the second and the third coordinates induces an elliptic fibration on $X$. Then by a result of Kharlamov (see \cite{Ak-Ma}), Viro's conjecture is true for $X$. 
\section{Double covering of certain blow-ups of $(\CP^1)^2$}
\label{doublecovering}
We describe a method of construction of real algebraic surfaces in $(\CP^1)^3$ of tridegree $(2k,2l,2)$ for any $(k,l)$ with $k\geq 1$ and $l\geq 1$. Consider a real algebraic surface $Z$ in $(\CP^1)^3$ defined by 
the polynomial
$
P^2+\varepsilon Q,
$
where $P$ is a real polynomial of tridegree $(k,l,1)$, the polynomial $Q$ is a real polynomial of tridegree $(2k,2l,2)$ and $\varepsilon$ is some small positive parameter. If $\lbrace P=0\rbrace$ and $\lbrace Q=0\rbrace$ are nonsingular and intersect transversely, the surface $Z$ is also nonsingular, and it is a small deformation of the double covering of $\lbrace P=0 \rbrace$ ramified along $\lbrace P=0 \rbrace \cap \lbrace Q=0 \rbrace$. More precisely, the surface $Z$ is obtained from an elementary equivariant deformation of the subvariety 
$$
Z_0=\lbrace U^{2}+Q=0 \: , \: P=0 \rbrace\subset(\C^*)^4
$$
compactified in the toric variety associated to the cone with vertex $(0,0,0,2)$ over the parallelepiped
$$
Conv\left((2k,0,0),(0,2l,0),(2k,2l,0),(0,0,2),(2k,0,2),(0,2l,2),(2k,2l,2)\right).
$$
This deformation is obtained via considering the family
$$
Z_t=\lbrace U^{2}+Q=0 \: , \: P=tU \rbrace,
$$
for $0\leq t\leq \sqrt{\varepsilon}$.
The real part $\R Z$ is homeomorphic to the disjoint union of two copies of $(\RP^1)^3\cap\lbrace P=0\rbrace\cap\lbrace Q\leq 0\rbrace$ attached to each other by the identity map of $(\RP^1)^3\cap\lbrace P=0\rbrace\cap\lbrace Q=0\rbrace$. The polynomials $P$ can be written in the following form:
$$
P(u_i,v_i)=v_3 P_1(u_i,v_i)+
u_3 P_0(u_i,v_i),
$$
where $P_0$, $P_1$ are homogeneous polynomials of degree $k$ in $u_1,v_1$ and of degree $l$ in $u_2,v_2$. As explained in Section \ref{lowdegree}, the surface $\lbrace P=0\rbrace$ is the blow-up of $(\CP^1)^2$ at the $2kl$ points of intersections of $\lbrace P_0=0\rbrace$ and $\lbrace P_1=0\rbrace$. Consider the real algebraic curve $\lbrace P=0\rbrace\cap\lbrace Q=0\rbrace$ in $(\CP^1)^3$ and consider its image $D$ under the projection $\pi:(\CP^1)^3\rightarrow (\CP^1)^2$ forgetting the last factor. Let us compute the bidegree of $D$. One can see that the intersection of $\lbrace P=0 \rbrace $ with $\CP^1\times [0:1]\times \CP^1$ (resp., $[0:1]\times\CP^1\times \CP^1$) is a curve of bidegree $(k,1)$ (resp., $(l,1)$). The intersection of $\lbrace Q=0 \rbrace $ with $\CP^1\times [0:1]\times \CP^1$ (resp., $[0:1]\times\CP^1\times \CP^1$) is a curve of bidegree $(2k,2)$ (resp., $(2l,2)$). Then the real algebraic curve $\lbrace P=0\rbrace \cap \lbrace Q=0\rbrace$ intersects $\CP^1\times [0:1]\times \CP^1$ (resp., $ [0:1]\times\CP^1\times \CP^1$) in $4k$ points (resp., $4l$ points). Considering the projection $\pi$, one concludes that the real algebraic curve $D$ is of bidegree $(4k,4l)$. For $1\leq i\leq 2kl$, denote by $L_i$ the exceptional lines of $\lbrace P=0 \rbrace $ corresponding to the intersection points of $\lbrace P_0=0\rbrace\cap\lbrace P_1=0\rbrace$. Since the real algebraic surface $\lbrace Q=0 \rbrace$ is of tridegree $(2k,2l,2)$, it intersects any exceptional line $L_i$ in exactly $2$ points. Considering the projection $\pi$, one concludes that the real algebraic curve $D$ has $2kl$ double points, one double point at each intersection point of $\lbrace P_0=0\rbrace$ and $\lbrace P_1=0\rbrace$.
Conversely, one has the following proposition.
\begin{proposition}
\label{generalposition}
Let $\lbrace P_0=0\rbrace$ and $\lbrace P_1=0\rbrace$ be two nonsingular real algebraic curves of bidegree $(k,l)$ in $(\CP^1)^2$ intersecting transversely in $2kl$ points. Let $D$ be an irreducible real algebraic curve of bidegree $(4k,4l)$ in $(\CP^1)^2$ with $2kl$ ordinary double points, one double point at each intersection point of $\lbrace P_0=0\rbrace $ and $\lbrace P_1=0\rbrace$. Consider the blow-up of $(\CP^{1})^{2}$ at the $2kl$ points of intersections of $\lbrace P_0=0\rbrace$ and $\lbrace P_1=0\rbrace$, given by the polynomial
$$
P(u_i,v_i)=v_3 P_1(u_i,v_i)+
u_3 P_0(u_i,v_i).
$$
%Then there exist $Q_0$, $Q_1$ and $Q_2$, polynomials of degree $(2k,2l)$ such that
%$$
%R=Q_2P_0^2-Q_1P_0P_1+Q_0P_1^2.
%$$
Then, there exists a real algebraic surface $\lbrace Q=0 \rbrace$ of tridegree $(2k,2l,2)$ in $(\CP^1)^3$ such that the strict transform of $D$ in $\lbrace P=0 \rbrace$ is given by the intersection $\lbrace P=0\rbrace\cap \lbrace Q=0\rbrace$.
%and
%$$
%Q(X_i,Y_i,Z_i)=Z_1^{2}Q_2(X_i,Y_i)+Z_0Z_1Q_1(X_i,Y_i)+
%Z_0^{2}Q_0(X_i,Y_i).
%$$
\end{proposition}
\begin{proof}
Denote by $x_1,...,x_{2kl}$ the intersection points of $\lbrace P_0=0 \rbrace$ and $\lbrace P_1=0 \rbrace$. Denote by $\mathcal{A}$ the linear system of curves of bidegree $(4k,4l)$ in $(\CP^1)^2$ with a singularity at each $x_i$.  We first show that the space $\mathcal{A}$ is of codimension $6kl$ in the space of curves of bidegree $(4k,4l)$ in $(\CP^1)^2$. Denote by $C_0$ the curve $\lbrace P_0=0 \rbrace$. One has the following exact sequence of sheaves:
$$
0\rightarrow\mathcal{O}_{\PP^1\times\PP^1}(D-C_0)\rightarrow
\mathcal{O}_{\PP^1\times\PP^1}(D)\rightarrow\mathcal{O}_{C_0}(D\mid_{C_0})\rightarrow 0,
$$
where $\mathcal{O}_{\PP^1\times\PP^1}(D-C_0)$ is the invertible sheaf associated to the divisor $D-C_0$, the sheaf $\mathcal{O}_{\PP^1\times\PP^1}(D)$ is the invertible sheaf associated to the divisor $D$, and 
$\mathcal{O}_{C_0}(D\mid_{C_0})$ denotes the restriction of $\mathcal{O}_{\PP^1\times\PP^1}(D)$ to $C_0$. 
Since the divisor $D-C_0$ is of bidegree $(3k,3l)$, for any $k>0$ and $l>0$ the invertible sheaf $\mathcal{O}_{\PP^1\times\PP^1}(D-C_0)$ is generated by its sections and one has $H^1(\PP^1\times\PP^1,\mathcal{O}_{\PP^1\times\PP^1}(D-C_0))=0$ (see for example \cite{Fulton}). Thus, the long exact sequence in cohomology associated to the above exact sequence splits. The first part of the long exact sequence is the following:
$$
0\rightarrow H^0(\PP^1\times\PP^1,\mathcal{O}(D-C_0))\rightarrow H^0(\PP^1\times\PP^1,\mathcal{O}(D))\overset{r}{\rightarrow}H^0(C_0,\mathcal{O}(D\mid_{C_0}))\rightarrow 0,
$$
where $r$ is the restriction map.
Denote by 
$$
E(x_1,\cdots ,x_{2kl})\subset H^0(\PP^1\times\PP^1,\mathcal{O}(D))
$$ 
the set of sections vanishing at least at order $2$ at $x_1,\cdots, x_{2kl}$. Then 
$$
\mathcal{A}=\PP(E(x_1,\cdots,x_{2kl})).
$$
Put $F(x_1,\cdots ,x_{2kl})=r(E(x_1,\cdots ,x_{2kl}))$. Denote by 
$$
G(x_1,\cdots ,x_{2kl})\subset H^0(\PP^1\times\PP^1,\mathcal{O}(D-C_0))
$$ 
the set of sections vanishing at $x_1,\cdots ,x_{2kl}$. One has then the following exact sequence:
$$
0\rightarrow G(x_1,\cdots ,x_{2kl})\rightarrow E(x_1,\cdots ,x_{2kl})\rightarrow F(x_1,\cdots ,x_{2kl})\rightarrow 0.
$$
Define on $C_0$ the divisor
$$
D'=D\cap C_0-E,
$$
where
$$
E=\left\lbrace 2\sum_{i=1}^{2kl}x_i\right\rbrace.
$$
By definition of $r$, the set $F(x_1,\cdots, x_{2kl})$ is the subspace of $H^0(C_0,\mathcal{O}_{C_0}(D\cap C_0))$ of sections vanishing at least at order $2$ at $x_1,\cdots ,x_{2kl}$. Consider the exact sheaf sequence
$$
0\rightarrow\mathcal{O}_{C_0}(D')\rightarrow\mathcal{O}_{C_0}(D\cap C_0)\rightarrow\mathcal{O}_{E}(D\cap C_0\mid_E)\rightarrow 0,
$$
and consider the associated long exact sequence
$$
0\rightarrow H^0(C_0,\mathcal{O}_{C_0}(D'))\rightarrow H^0(C_0,\mathcal{O}_{C_0}(D\cap C_0))\overset{r'}{\rightarrow}H^0(E,\mathcal{O}(D\cap C_0\mid_{E}))\rightarrow \cdots
$$
Thus, one sees that $F(x_1,\cdots, x_{2kl})=\ker r'$ is identified with $H^0(C_0,\mathcal{O}_{C_0}(D'))$.
\\Let us compute $h^0(C_0,D')$. The divisor $D'$ is of degree $4kl$. Moreover, one has
$$
\begin{array}{c@{=\mbox{ }}c}
\deg(K_{C_0}-D')\mbox{ } & -\deg(D')+\deg(K_{C_0})\\
& -\deg(D')-2+2g(C_0)\\
& -4kl-2+2(k-1)(l-1)\\
& -2kl-2k-2l.\\
\end{array}
$$
But $-2kl-2k-2l<0$, so $h^0(C_0,K_{C_0}-D')=0$, and by Riemann-Roch formula, one gets
$$
\begin{array}{c@{=\mbox{ }}c}
h^0(C_0,D')\mbox{ } & \deg(D')+1-g(C_0)\\
 & 4kl+1-(k-1)(l-1)\\
 & 3kl+k+l.
\end{array}
$$
Therefore, one has
$$
\dim(E(x_1,...,x_{2kl}))=3kl+k+l+\dim(G(x_1,...,x_{2kl})).
$$
Now, compute $\dim(G(x_1,...,x_{2kl}).$ Consider the following exact sequence:
$$
0\rightarrow\mathcal{O}_{\PP^1\times\PP^1}(D-C_0^2)\rightarrow
\mathcal{O}_{\PP^1\times\PP^1}(D-C_0)\rightarrow\mathcal{O}_{C_0}((D-C_0)\mid_{C_0})\rightarrow 0.
$$
Passing to the long exact sequence, one obtains:
$$
0\rightarrow H^0(\PP^1\times\PP^1,\mathcal{O}(D-C_0^2))\rightarrow G(x_1,...,x_{2kl})\rightarrow r(G(x_1,...,x_{2kl}))\rightarrow 0.
$$
With a similar computation as before, one sees that 
$$
\begin{array}{c@{=\mbox{ }}c}
\dim(r(G(x_1,...,x_{2kl}))\mbox{ } & 4kl+1-(k-1)(l-1)\\
 & 3kl+k+l.\\
\end{array}
$$
On the other hand, $h^0(\PP^1\times\PP^1,D-C_0^2)=(2k+1)(2l+1)-1$.
So finally,
$$
\begin{array}{c@{=\mbox{ }}c}
\dim(E(x_1,...,x_{2kl}))\mbox{ } & 3kl+k+l+3kl+k+l+4kl+2k+2l\\
 & 10kl+4k+4l,
\end{array}
$$
and $ \textrm{codim}(\mathcal{A})=6kl$. Denote by $\tilde{\mathcal{A}}$ the linear system of curves obtained as the proper transform of the linear system $\mathcal{A}$ and denote by $\mathcal{B}$ the linear system of surfaces of tridegree $(2k,2l,2)$ in $(\CP^{1})^{3}$. By restriction to the surface $\lbrace P=0\rbrace$, one obtains a map from $\mathcal{B}$ to $\tilde{\mathcal{A}}$. By definition, the kernel of this map correspond to the linear system of surfaces of tridegree $(k,l,1)$. Then the dimension of the image of the restriction map is
$$
3(2k+1)(2l+1)-2(k+1)(l+1)-1=10kl+4(k+l),
$$ 
and the restriction map from $\mathcal{B}$ to $\tilde{\mathcal{A}}$ is surjective. 
%the space of curves $\lbrace C=0\rbrace$ of degree $(4k,4l)$ such that there exist $Q_0$, $Q_1$ and $Q_2$  with
%$$
%C=Q_2P_0^2-Q_1P_0P_1+Q_0P_1^2.
%$$
%One has obviously $\mathcal{B}\subset\mathcal{A}$. Compute the dimension of $\mathcal{B}$. Remark that 
%$$
%Q_2P_0^2-Q_1P_0P_1+Q_0P_1^2=Q'_2P_0^2-Q'_1P_0P_1+Q'_0P_1^2
%$$
%if and only if there exist $A$ and $B$ polynomials of degree $(k,l)$ such that
%$$
%\left\lbrace\begin{array}{cc}
%Q_0'=&Q_0+AP_0,\\
%Q'_2=&Q_2+BP_1,\\
%Q'_1=&Q_1+AP_1+BP_0.\\
%\end{array}\right.
%$$
%Then
%$$
%\dim(\mathcal{B})=3(2k+1)(2l+1)-2(k+1)(l+1)-1=10kl+4k+4l,
%$$
%and $\mathcal{A}=\mathcal{B}$.
\end{proof}
In summary, we obtain the following method of construction of real algebraic surfaces of degree $(2k,2l,2)$ in $(\CP^1)^3$.
\begin{enumerate}
\item Consider two nonsingular real algebraic curves $\lbrace P_0=0 \rbrace$ and $\lbrace P_1=0 \rbrace$ of bidegree $(k,l)$ in $(\CP^1)^2$ intersecting transversely.
\item Consider an irreducible real algebraic curve $\lbrace R=0\rbrace $ of bidegree $(4k,4l)$ in $(\CP^1)^2$ with $2kl$ double points, one double point at each intersection point of $\lbrace P_0=0 \rbrace$ and $\lbrace P_1=0 \rbrace$. 
\item Consider the polynomial $P(u_i,v_i)=v_3P_1(u_i,v_i)+u_3P_0(u_i,v_i)$. The real algebraic surface $X$ in $(\CP^1)^3$ defined by the polynomial $P$ is the blow-up of $(\CP^1)^2$ at the $2kl$ points of intersection of $\lbrace P_0=0 \rbrace$ and $\lbrace P_1=0 \rbrace$. Consider the strict transform $\tilde{\mathcal{C}}$ of $\lbrace R=0 \rbrace$ under this blow-up. By Proposition \ref{generalposition}, there exists a polynomial $Q$ of tridegree $(2k,2l,2)$ such that 
$$
\tilde{\mathcal{C}}=\lbrace P=0 \rbrace\cap \lbrace Q=0 \rbrace.
$$
Denote by $X_-\subset \R X$ the subset which projects to $\lbrace R\leq 0\rbrace$.
\item Consider the surface 
$$
Z=\lbrace P^2+\varepsilon Q=0\rbrace,
$$ for $\varepsilon>0$ small enough. Then, the surface $Z$ is a nonsingular real algebraic surface of tridegree $(2k,2l,2)$ in $(\CP^1)^3$ and its real part $\R Z$ is homeomorphic to the disjoint union of two copies of $X_-$ attached to each other by the identity map of $\R \tilde{\mathcal{C}}$.
\end{enumerate}
\section{A family of real algebraic surfaces with asymptotically maximal number of handles in $(\CP^1)^3$}
\label{constructionasympt}
Let $k\geq 2$ and $l\geq 2$. To prove Theorem \ref{theorem1}, we construct a real algebraic curve of bidegree $(4k,4l)$ in $(\CP^1)^2$ with $2kl$ double points which are the intersection points of two real algebraic curves of bidegree $(k,l)$ in $(\CP^1)^2$. The main difficulty is that these $2kl$ double points have to be in a special position. In fact, there is no curve of bidegree $(k,l)$ passing through $2kl$ points in $(\CP^1)^2$ in general position. In ~\cite{Brugalle2006}, Brugall\'e constructed a family of reducible curves 
$
\lbrace C^1_n=0 \rbrace \cup \lbrace C^2_n=0 \rbrace
$ in the $n$th Hirzebruch surface $\Sigma_n$, where $ \lbrace C^1_n=0 \rbrace$ has Newton polytope $Conv((0,0),(n,0),(0,1))$ and $\lbrace C^2_n=0 \rbrace$ has Newton polytope $Conv((0,0),(n,0),(0,2),(n,1))$. The chart of $\lbrace C^1_n=0 \rbrace \cup \lbrace C^2_n=0 \rbrace$ is depicted in Figure \ref{courbebrugalle}.
\begin{figure}[h!]
\centerline{
\includegraphics[width=9cm,height=5cm]{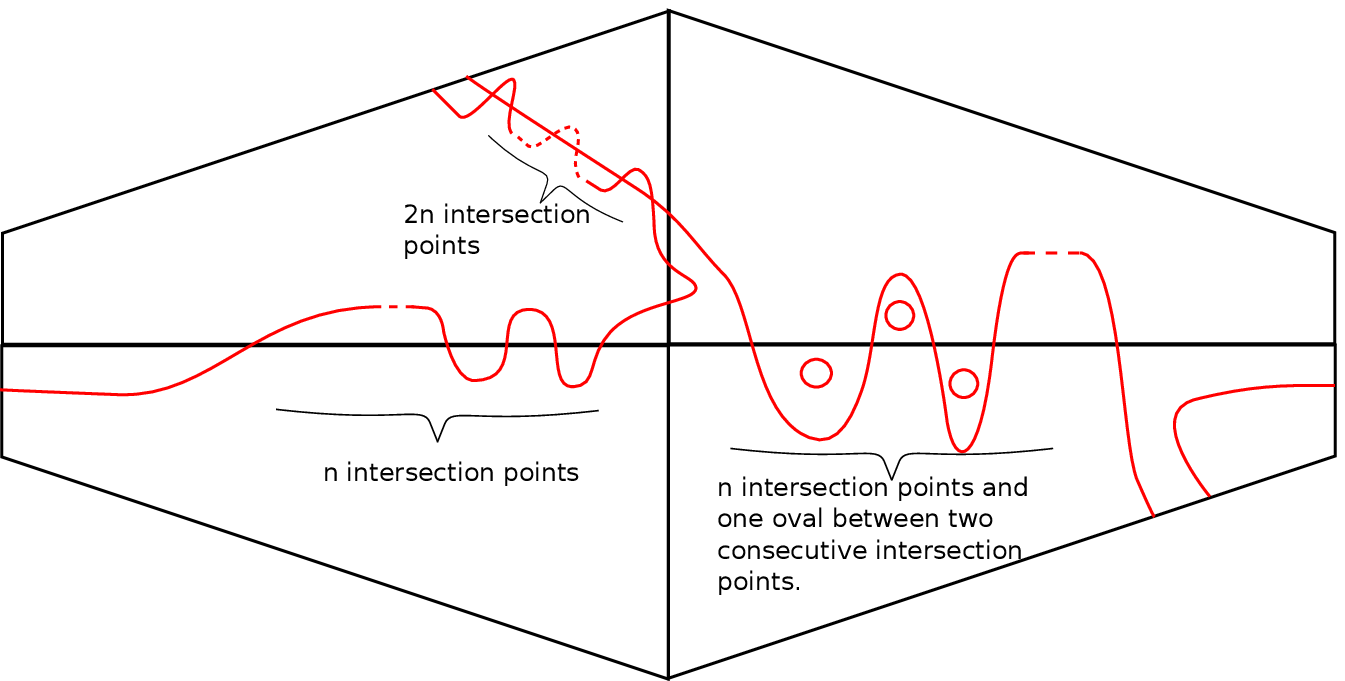}}
\setlength\abovecaptionskip{0cm}
\caption{The chart of $\lbrace C^1_n=0 \rbrace \cup \lbrace C^2_n=0 \rbrace $}
\label{courbebrugalle}
\end{figure}
Using Brusotti theorem (see Theorem \ref{Brusottitheorem} or \cite{Risler}), perturb the curve 
$$ 
\lbrace C^1_{2k-1}=0 \rbrace \cup \lbrace C^2_{2k-1}=0\rbrace
$$ keeping $k$ double points, as depicted in Figure \ref{courbeperturbee}. 
\begin{figure}[h!]
\centerline{
\includegraphics[width=9cm,height=5cm]{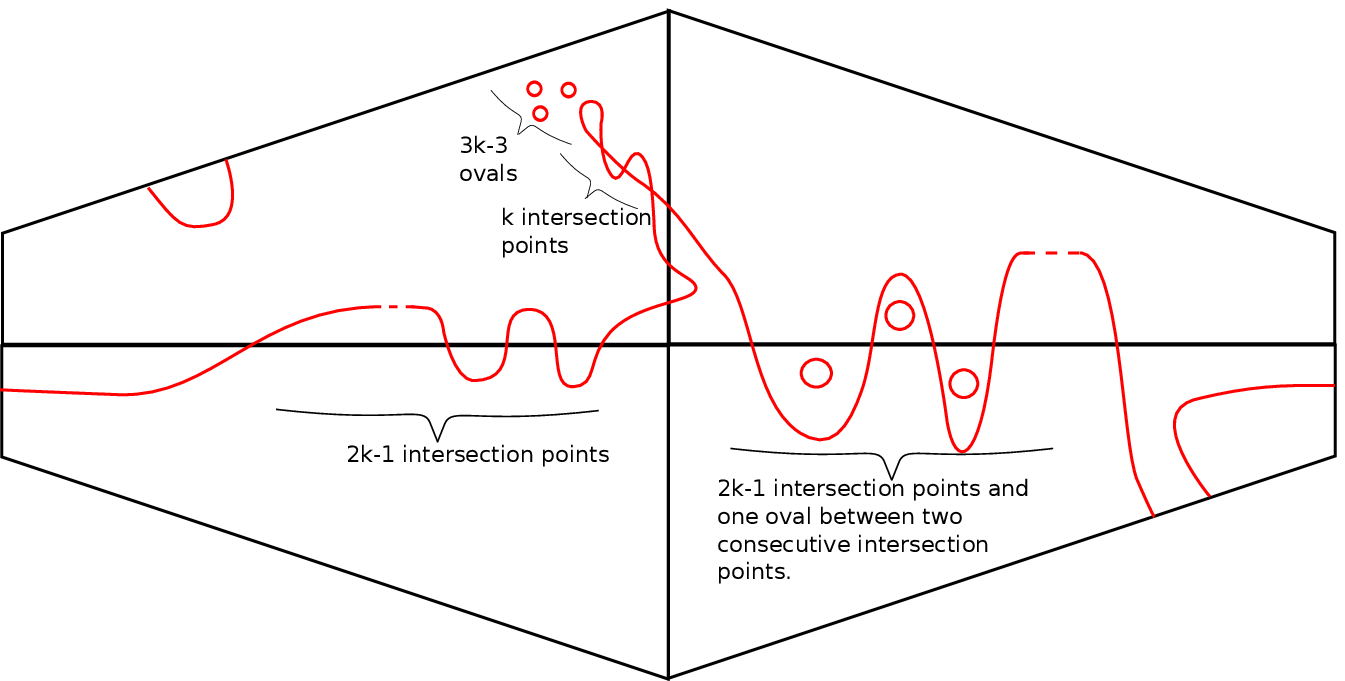}}
\setlength\abovecaptionskip{0cm}
\caption{The chart of $\lbrace C_{2k-1}=0 \rbrace$}
\label{courbeperturbee}
\end{figure}
Denote the resulting curve by \linebreak $\lbrace C_{2k-1}=0 \rbrace$, and by $c_{i,j}$  the coefficient of the monomial $(i,j)$ in the polynomial $C_{2k-1}$.
Since the edge $[(0,3),(4k-2,1)]$ of the Newton polytope of $C_{2k-1}$ is of length $2$, one can assume, up to a linear change of coordinates, that $c_{0,3}=c_{4k-2,1}$. Denote by $\Gamma$ the rectangle $Conv((0,0),(0,4),(4k,0),(4k,4))$ and consider the charts of the polynomials $xC_{2k-1}(x,y)$ and $x^{4k-1}y^4C_{2k-1}(\frac{1}{x},\frac{1}{y})$. Complete the rectangle $\Gamma$ with other charts of polynomials, as depicted in Figure \ref{courbebase}. By Shustin's patchworking theorem for curves with double points, there exists a polynomial $P$ of Newton polytope $\Gamma$ whose chart is depicted in Figure \ref{courbebase}. 
%Assume that $c_{0,0}<0$. Introduce the following polynomials:
%\begin{itemize}
%\item $A(X,Y)=c_{0,0}+c_{0,1}Y+c_{0,2}Y^2+c_{0,3}Y^3-X$,
%\item $B(X,Y)=-c_{0,0}-c_{0,1}Y-c_{0,2}Y^2-Y^3+c_{0,3}X$,
%\item $D(X,Y)=-c_{0,0}+c_{0,3}X+Y$,
%\item $E(X,Y)=c_{4k-2,0}\left(\frac{c_{0,3}}{c_{4k-2,1}}\right)^{\frac{4}{3}}+X+c_{0,3}Y$,
%\item $A'(X,Y)=c_{0,0}+c_{0,1}\left(\frac{c_{4k-2,1}}{c_{0,3}}\right)^{\frac{1}{3}}Y+c_{0,2}\left(\frac{c_{4k-2,1}}{c_{0,3}}\right)^{\frac{2}{3}}Y^2+c_{4k-2,1}Y^3-X$,
%\item $B'(X,Y)=-c_{0,0}-c_{0,1}\left(\frac{c_{4k-2,1}}{c_{0,3}}\right)^{\frac{1}{3}}Y-c_{0,2}\left(\frac{c_{4k-2,1}}{c_{0,3}}\right)^{\frac{2}{3}}Y^2-Y^3+c_{4k-2,1}X$,
%\item $D'(X,Y)=-c_{0,0}+c_{4k-2,1}X+Y$,
%\item $E'(X,Y)=c_{4k-2,0}+X+c_{4k-2,1}Y$.
%\end{itemize}
%Denote by $\Gamma$ the rectangle $Conv((0,0),(0,4),(4k,0),(4k,4))$, and glue the charts of the polynomials
%\begin{itemize}
%\item $XC'_{2k-1}(X,Y)$,
%\item $XA(\frac{1}{X},Y)$, $Y^3B(X,\frac{1}{Y})$,  $Y^3D(X,Y)$ $XY^4E(\frac{1}{X},\frac{1}{Y})$,
%\item $X^{4k-1}Y^4C'_{2k-1}\left(\frac{1}{\left(\frac{c_{4k-2,1}}{c_{0,3}}\right)^\frac{2}{6k-3}X},\frac{1}{\left(\frac{c_{0,3}}{c_{4k-2,1}}\right)^\frac{1}{3}Y}\right)$,
%\item $X^{4k-1}Y^4A'(X,\frac{1}{Y})$, $X^{4k}YB'(X,\frac{1}{Y})$, $X^{4k-1}D'(X,Y)$, $YX^{4k}E'(\frac{1}{X},\frac{1}{Y})$.
%\end{itemize}
\begin{figure}[h!]
\centerline{
\includegraphics[width=9cm,height=5cm]{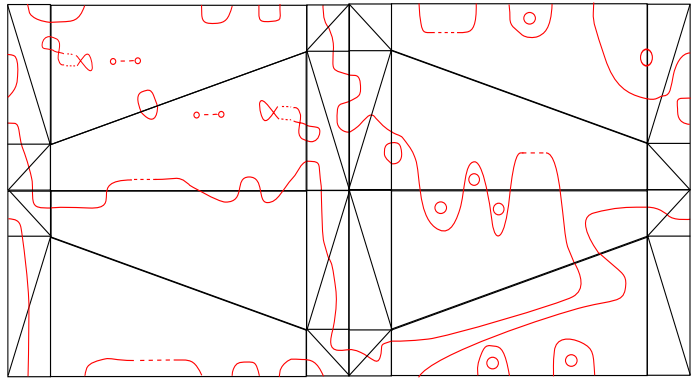}}
\setlength\abovecaptionskip{0cm}
\caption{The chart of $P$}
\label{courbebase}
\end{figure}
Denote by $(x_1,y_1),\cdots,(x_{2k},y_{2k})$ the coordinates of the $2k$ double points of $ \lbrace P=0 \rbrace $. These $2k$ double points are on the intersection of two algebraic curves of bidegree $(k,1)$ in $(\CP^1)^2$, but it could happen that these two curves are reducible. It turns out that to prove Theorem \ref{theoremprincipal}, it is important to have the $2k$ double points of $\lbrace P=0\rbrace$ on the intersection of two irreducible curves of bidegree $(k,1)$ in $(\CP^1)^2$, which is the case if the $2k$ double points of $\lbrace P=0 \rbrace $ are in general position.
\begin{lemma}
\label{generalpositionlemma}
One can perturb the polynomial $P$ so that the double points of $\lbrace P=0 \rbrace $ are in general position.
\end{lemma}

\begin{proof} 
 Denote by $\mathcal{E}$ the set of polynomials of bidegree $(4k,4)$, and denote by $N$ its dimension. Denote by $S$ the subset of $\mathcal{E}$ consisting of polynomials defining curves which have double points in a neighborhood of $\lbrace(x_1,y_1)$,$\cdots$,$(x_{2k},y_{2k})\rbrace$. By the Brusotti theorem (see Theorem \ref{Brusottitheorem} in Section \ref{Transversality}), there exists a small neighborhood $U$ of $P$ in $\mathcal{E}$ such that $S\cap U$ is a transverse intersection of $2k$ hypersurfaces in $\mathcal{E}$. Then, $\dim(S\cap U)=N-2k$. Define the incidence variety $\mathcal{I}$ associated to $S\cap U$ by 
$$ 
\mathcal{I}=\left\lbrace (Q,z_1,...,z_{2k})\in \left( S\cap U \right) \times ((\C^*)^2)^{2k} \mid z_i \mbox{ is a double point of } Q\right\rbrace. 
$$
One has
$\pi_1(\mathcal{I})= S\cap U$, where 
$$
\pi_1:\mathcal{E}\times((\C^*)^2)^{2k}\rightarrow\mathcal{E}
$$
denotes the first projection.
Denote by $\pi_2$ the second projection:
$$
\pi_2:\mathcal{E}\times((\C^*)^2)^{2k}\rightarrow ((\C^*)^2)^{2k}.
$$
To prove the lemma, it is enough to show that $\dim(\pi_2(\mathcal{I}))=4k$.
By the Brusotti theorem, $\pi_1$ induces a local homeomorphism from $\mathcal{I}$ to $S\cap U$, so $\dim(\mathcal{I})=\dim(S\cap U)=N-2k$. By Lemma \ref{theoremRR}, one has 
$$
\dim(\pi_2^{-1}(x)\cap \mathcal{I})=N-6k,
$$
for any $x\in \pi_2(\mathcal{I})$.
So $\dim(\pi_2(\mathcal{I}))=N-2k-(N-6k)=4k$. 
\end{proof}
Hence we can assume that the double points of $P$ are the intersection points of two irreducible curves of bidegree $(k,1)$ in $(\CP^1)^2$.
Denote by $L(x,y)=0$ and $M(x,y)=0$ the equations of two distinct irreducible curves of bidegree $(k,1)$ in $(\CP^1)^2$ passing through $(x_1,y_1),\cdots,(x_{2k},y_{2k})$, the double points of $\lbrace P=0 \rbrace$.
%\begin{figure}[h!]
%\centerline{
%\includegraphics[width=15cm,height=10cm]{empilement.eps}}
%\setlength\abovecaptionskip{0cm}
%\caption{The Newton polytopes of $P_1^0$, $P_1^1$, $L_1^0$ and $M_1^0$}
%\label{empilement}
%\end{figure}
%\begin{lemma}
%\label{lemmaprep}
%There exist three real algebraic curves in $(\CP^1)^2$ of equations $(P_1=0)$, $(L_1=0)$ and $(M_1=0)$ such that:
%\begin{itemize}
%\item $P_1$ is of bidegree $(4k,3)$, $L_1$ and $M_1$ are of bidegree $(k,1)$.
%\item The restriction of $P_1$ to the edge $[(0,3),(4k,3)]$ is equal to the restriction of $P$ to the edge $[(0,0),(4k,0)]$. 
%\item The restriction of $L_1$ to the edge $[(0,1),(k,1)]$ is equal to the restriction of $L$ to the edge $[(0,0),(k,0)]$.
%\item The restriction of $M_1$ to the edge $[(0,1),(k,1)]$ is equal to the restriction of $M$ to the edge $[(0,0),(k,0)]$.
%\item The double points of $\left(P_1=0\right)$ coincide with the intersection points of $\left(L_1=0\right)$ and $\left(M_1=0\right)$.   
%\end{itemize}
%\end{lemma}
\\For $1\leq h\leq l$, put
\begin{itemize}
\item
$
L_h(x,y)=y^{h}L(x,\frac{1}{y}),
$
\item
$
M_h(x,y)=y^{h}M(x,\frac{1}{y})$, 
\end{itemize}
if $h$ is odd, and 
\begin{itemize}
\item
$
L_h(x,y)=y^{h-1}L(x,y),
$
\item
$
M_h(x,y)=y^{h-1}M(x,y)$, 
\end{itemize}
if $h$ is even.

Choose an irreducible polynomial $P_1$ of bidegree $(4k,2)$ such that the curve $\lbrace P_1=0\rbrace$ has a double point at all the $(x_i,\frac{1}{y_i})$. Let $P_1^0$ be the polynomial of bidegree $(4k,1)$ such that the restriction of $P_1^0$ to the edge $[(0,0),(4k,0)]$ coincide with the restriction of $P_1$ to the edge $[(0,2),(4k,2)]$ and the restriction of $P_1^0$ to the edge $[(0,1),(4k,1)]$ coincide with the restriction of $P$ to the edge $[(0,0),(4k,0)]$. Put $P_1^1=y^2 P_1^0$ (see Figure \ref{empilement_1}).
\\For $2\leq h\leq l$, put
\begin{itemize}
\item
$
P_h(x,y)=y^{4h-1}P(x,\frac{1}{y})$, if $h$ is odd, and 

\item
$
P_h(x,y)=y^{4h-5}P(x,y)$, if $h$ is even.
\end{itemize}

Let $P_{l+1}^0$ be a polynomial of bidegree $(4k,1)$ such that the restriction of $P_{l+1}^0$ to the edge $[(0,0),(4k,0)]$ is equal to the restriction of $P_{l}$ to the edge $[(0,4l-1),(4k,4l-1)]$. Put $P_{l+1}=y^{4l-1}P_l^0$ ( see Figure \ref{empilement_1}).
\begin{figure}[h!]
\centerline{
\includegraphics[width=12cm,height=7cm]{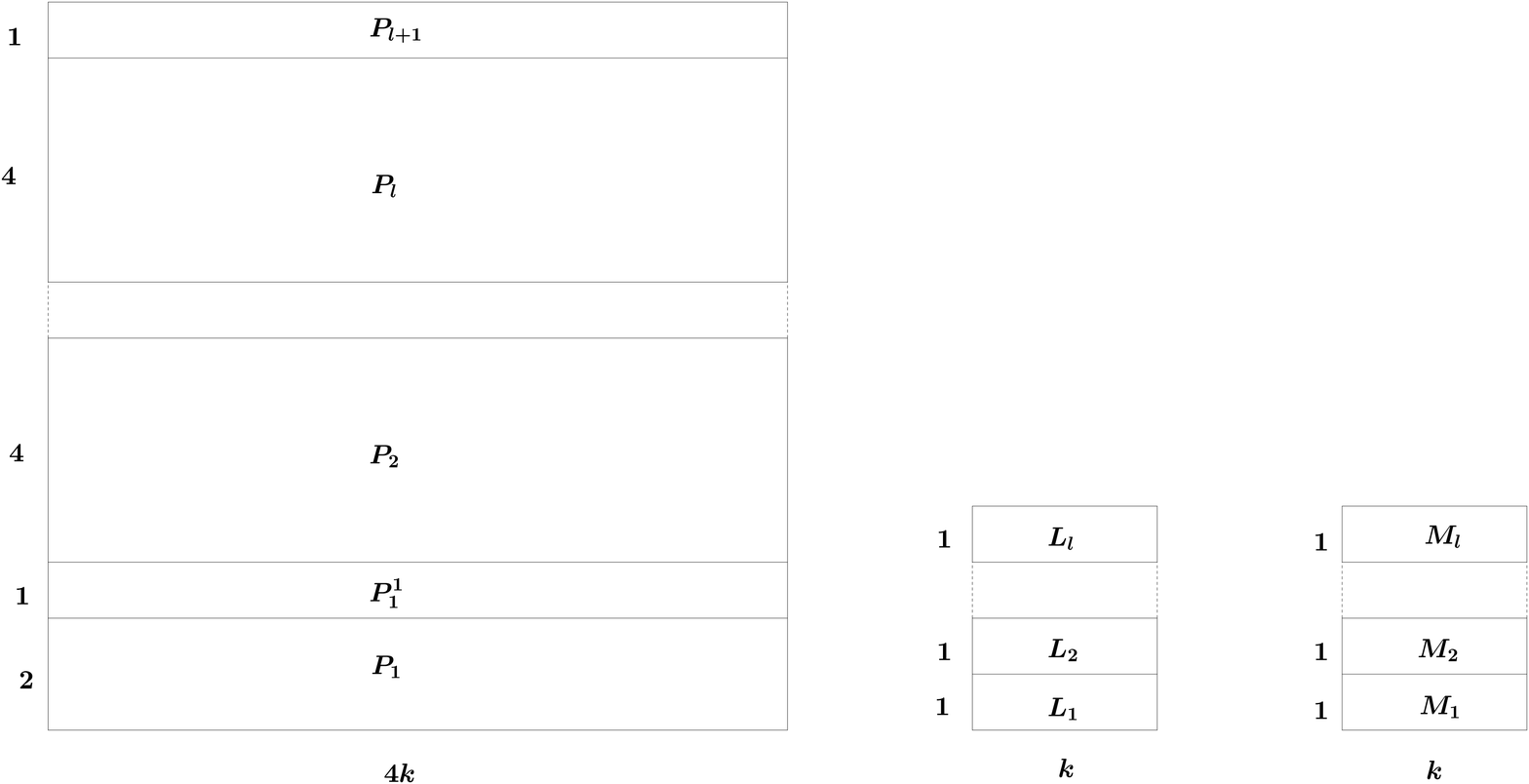}}
\setlength\abovecaptionskip{0cm}
\caption{The Newton polytopes of polynomials $P_j$, $L_j$ and $M_j$.}
\label{empilement_1}
\end{figure}
\begin{theorem}
\label{theoremprincipal}
There exist three real algebraic curves $\lbrace \tilde{P}=0 \rbrace$, $ \lbrace \tilde{L}=0 \rbrace$ and $\lbrace \tilde{M}=0 \rbrace$ in $(\CP^{1})^{2}$ such that:
\begin{itemize} 
\item The polynomial $\tilde{P}$ is of bidegree $(4k,4l)$, the polynomials $\tilde{L}$ and $\tilde{M}$ are of bidegree $(k,l)$.
\item
The chart of $\tilde{P}$ is homeomorphic to the gluing of the charts of the polynomials $P_1$, $P_1^1$ and $P_j$, for $2\leq j\leq l+1$. 
\item
The curve $\lbrace \tilde{P}=0\rbrace$ has $2kl$ double points,  one double point at each intersection point of $\lbrace \tilde{L}=0 \rbrace$ and $ \lbrace \tilde{M}=0 \rbrace$.
\end{itemize}
\end{theorem}
\textit{Proof of Theorem \ref{theorem1}.}
By the construction presented in Section \ref{doublecovering}, the three curves $\lbrace \tilde{P}=0 \rbrace$, $\lbrace \tilde{L}=0 \rbrace$ and $\lbrace \tilde{M}=0 \rbrace$ produce a real algebraic surface $X_{k,l}$ of tridegree $(2k,2l,2)$ in $(\CP^1)^3$. 
\begin{figure}[h!]
\centerline{
\includegraphics[width=10cm,height=10cm]{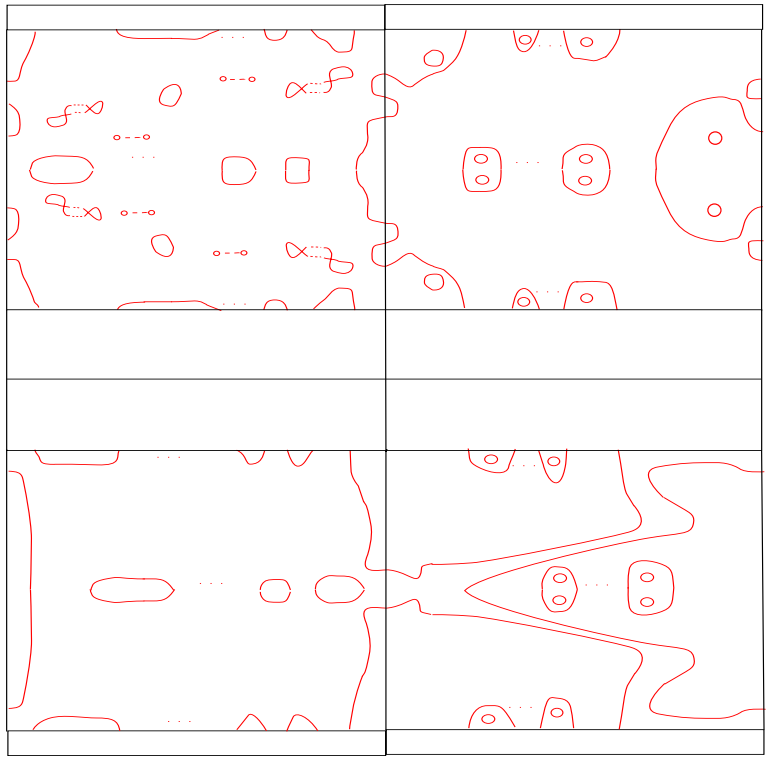}}
\setlength\abovecaptionskip{0cm}
\caption{A part of the chart of $\tilde{P}$ in the case $l=3$.}
\label{courbebaserecolle}
\end{figure}
The sign of $\tilde{P}$ is the same in any empty oval of $\lbrace \tilde{P}=0 \rbrace $ coming from the gluing of the charts of the polynomials $P_h$, for $2\leq h\leq l$. Assume that this sign is positive. Denote by $Y_+$ (resp., $Y_-$) the subset of $\R Y$ which projects to $\lbrace \tilde{P}\geq 0 \rbrace$ (resp., $\lbrace \tilde{P}\leq 0 \rbrace$), where $Y$ is $(\CP^1)^2$ blown up at the $2kl$ double points of $\lbrace \tilde{P}=0 \rbrace$. Then, $\R X_{k,l}$ is homeomorphic to the disjoint union of two copies of $Y_-$ attached by the identity map of the strict transform of $\lbrace \tilde{P}=0 \rbrace$. We depicted a part of the chart of $\tilde{P}$ in the case $l=3$ in Figure \ref{courbebaserecolle}. By counting all the ovals of $\lbrace \tilde{P}=0 \rbrace$ containing two empty ovals, one sees that 
$$
b_0(Y_-)\geq (2k-2)(l-2).
$$
Let us estimate from below the Euler characteristic of $Y_+$. By counting empty ovals of $\lbrace \tilde{P}=0 \rbrace$, one sees that $Y_+$ contains at least $(6k-4)(l-1)+3(2k-2)(l-2)$ components homeomorphic to a disc. One sees also that $Y_+$ contains at least $l-2$ components homeomorphic to a disc with $k-1$ holes, and another component with at least $(k-1)(l-2)$ holes. Since the real part of the real algebraic curve curve $\lbrace \tilde{P}=0\rbrace$ contains at most $(4k-1)(4l-1)+1-2kl$ connected components, one sees that there exist $c_0,d_0,e_0\in\Z$ such that 
$$
\chi(Y_+)\geq 10kl+c_0 \cdot 2k+d_0 \cdot 2l+e_0.
$$
% The chart of $\tilde{P}$ contains at least $8kl+\alpha(k,l)$ empty ovals and $2kl+\beta(k,l)$ configurations of ovals of the type $1<2>$ distributed in $\gamma(k,l)$ connected components of the complementary of $(\tilde{P}=0)$, where $\alpha$, $\beta$ and $\gamma$ are some affine functions. 
Thus, one has
$$
b_0(\R X_{k,l})=b_0(Y_-)\geq 2kl-4k-2l,
$$
and
$$
\begin{array}{ccc}
\chi(\R X_{k,l}) & = & 2\chi(Y_-)\\
& = & 2(\chi(\R Y)-\chi(Y_+))\\
& \leq & 2(-2kl-10kl+c_0\cdot 2k+d_0\cdot 2l+e_0)\\
& \leq & -24kl+2c_0\cdot 2k+2d_0\cdot 2l+2e_0.
\end{array}
$$
Therefore, one obtains
$$
\begin{array}{ccc}
b_1(\R X_{k,l}) & = & 2b_0(\R X_{k,l})-\chi(\R X_{k,l})\\
& \geq & 28kl-c \cdot 2k-d \cdot 2l+e,
\end{array}
$$
where $c,d,e\in\Z$, which proves Theorem \ref{theorem1}. \qed
\\\textit{Proof of Theorem \ref{theoremprincipal}.}
%\\We glue together the polynomials $P_h$, $L_h$ and $M_h$ to obtain a curve of degree $(4k,4l)$ with $2kl$ double points lying on the intersection of two auxiliary curves of degree $(k,l)$. 
\\The construction follows the same lines as the proof of the main theorem in \cite{Shustin}. We recall the main steps referring at \cite{Shustin} for the proofs of auxiliary statements. Denote by $\Delta$ the rectangle $Conv\left((0,0),(4k,0),(0,4l),(4k,4l)\right)$, denote by $\Delta_1$ the Newton polytope of $P_1$, denote by $\Delta_1^1$ the Newton polytope of $P_1^1$ and denote by $\Delta_h$ the Newton polytope of $P_h$, for $2\leq h\leq l+1$. Denote by $\Lambda$ the rectangle $Conv\left((0,0),(k,0),(0,l),(k,l)\right)$ and denote by $\Lambda_h$ the Newton polytope of $L_h$, for $1\leq h\leq l$. Denote by $a_{i,j}$, for all $(i,j)\in\Delta$, the collection of real numbers satisfying
\begin{eqnarray*}
P_1^1=\sum_{i,j\in\Delta_1^1}a_{i,j}x^iy^j,\\
P_h=\sum_{i,j\in\Delta_h}a_{i,j}x^iy^j,
\end{eqnarray*}
for $1\leq h\leq l+1$. Denote by $b_{i,j}$, for all $(i,j)\in\Gamma$ and by $c_{i,j}$, for all $(i,j)\in\Gamma$, the collections of real numbers satisfying
\begin{eqnarray*}
L_h=\sum_{i,j\in\Lambda_h}b_{i,j}x^iy^j,\\
M_h=\sum_{i,j\in\Lambda_h}c_{i,j}x^iy^j,
\end{eqnarray*}
for $2\leq h\leq l+1$.
We look for the desired polynomials in a one-parametric family of polynomials.
\begin{eqnarray}
\label{1}
P_t=\sum_{i,j\in\Delta}A_{i,j}(t)x^iy^jt^{\nu_P(i,j)},\\\label{2}
L_t=\sum_{i,j\in\Lambda}B_{i,j}(t)x^iy^jt^{\nu_L(i,j)},\\
M_t=\sum_{i,j\in\Lambda}C_{i,j}(t)x^iy^jt^{\nu_M(i,j)},\label{3}
\end{eqnarray}
where 
\begin{itemize}
\item $\vert A_{i,j}(t)-a_{i,j}\vert \leq Kt$, 
\item $\vert B_{i,j}(t)-b_{i,j}\vert \leq Kt$, and
\item $\vert C_{i,j}(t)-c_{i,j}\vert \leq Kt$, 
\end{itemize}
for some positive constant $K$. The piecewise-linear functions $\nu_P$, $\nu_L$ and $\nu_M$ are defined as follows. The function $\nu_L$ is the piecewise-linear function independent of $i$ certifying the convexity of the decomposition $\Lambda=\cup_h \Lambda_h$, satisfying $\nu_L(0,1)=0$, of slope $-1$ on $\Lambda_1$ and of slope $h$ on $\Lambda_{h+1}$, for $1\leq h\leq l-1$. Put $\nu_M=\nu_L$. The function $\nu_P$ is the piecewise-linear function independent of $i$ certifying the convexity of the decomposition $\Delta=\cup_h \Delta_h$, satisfying $\nu_P(0,2)=0$, of slope $-1$ on $\Delta_1$, of slope $0$ on $\Delta_1^1$ and of slope $h$ on $\Delta_{h+1}$, for $1\leq h\leq l$. 
\\Denote by $\mu^h_L(j)=a_h+h j$ the affine function equal to $\nu_L$ on $\Lambda_h$, $h=1, \cdots ,l$. Denote by $\mu^h_P(j)=a'_h+h j$  the affine function equal to $\nu_P$ on $\Delta_h$, $h=1,\cdots ,l$. The substitution of $\nu_P^h=\nu_P-\mu_P^h$ for $\nu_P$ in (\ref{1}), the substitution of $\nu_L^h=\nu_L-\mu_L^h$ for $\nu_L$ in (\ref{2}) and the substitution of $\nu_M^h=\nu_M-\mu_M^h$ for $\nu_M$ in (\ref{3}) give the families
$$
P_{h,t}=P_h+\sum_{(i,j)\notin\Delta_h}A_{i,j}(t)x^iy^jt^{\nu_P^h(j)}+\sum_{(i,j)\in\Delta_h}(A_{i,j}(t)-a_{i,j})x^iy^j,
$$

$$
L_{h,t}=L_h+\sum_{(i,j)\notin\Lambda_h}B_{i,j}(t)x^iy^jt^{\nu_L^h(j)}+\sum_{(i,j)\in\Lambda_h}(B_{i,j}(t)-b_{i,j})x^iy^j,
$$

$$
M_{h,t}=M_h+\sum_{(i,j)\notin\Lambda_h}C_{i,j}(t)x^iy^jt^{\nu_M^h(j)}+\sum_{(i,j)\in\Lambda_h}(C_{i,j}(t)-c_{i,j})x^iy^j,
$$ 
for all $h=1,\cdots ,l$.
These substitutions are the composition of the coordinate change
$$
T_h(x,y)=(x,yt^h)
$$
with the multiplication of the polynomial by some positive number.
$$
P_{h,t}=t^{-a'_h}P_t(T_h^{-1}(x,y)),
$$ 
$$
L_{h,t}=t^{-a_h}L_t(T_h^{-1}(x,y)),
$$
$$
M_{h,t}=t^{-a_h}M_t(T_h^{-1}(x,y)).
$$
In particular, the point $(x,y)$ is a singular point of $P_t$ in $(\C^*)^2$ if and only if the point $T_h(x,y)$ is a singular point of $P_{h,t}$. 
%Define the polynomials $F_h=P_h L_h M_h$, $F_{h,t}=P_{h,t}L_{h,t}M_{h,t}$ and $F_t=P_t L_t M_t$.
\\Fix a compact $Q\subset(\C^*)^2$, whose interior contains all singular points of $P_h$ in $(\C^*)^2$, for $h=1,...,l$. Denote by $z_{h_p}$, for $p\in I_h$, the singular points of $P_{h,t}$ in $Q$.
\begin{lemma}$($see \cite{Shustin} $)$
There exists $t_0>0$ such that for any $t\in (0,t_0)$, the points $T_h(z_{h_p})$ for $p\in I_h$ and $h=1,...,l$, are the only singular points of $P_t$ in $(\C^*)^2$.
\end{lemma}
We define $A_{i,j}$, $B_{i,j}$ and $C_{i,j}$ as smooth functions of $t$ such that $A_{i,j}(0)=a_{i,j}$, \linebreak $B_{i,j}(0)=b_{i,j}$ and $C_{i,j}(0)=c_{i,j}$ and such that for any $h\in\lbrace 1,...,l\rbrace$, the polynomial $P_{h,t}$ has $2k$ double points in $Q$ which lie on the intersection of the curves $\lbrace L_{h,t}=0\rbrace$ and $\lbrace M_{h,t}=0\rbrace$. Following the notations of Section \ref{Transversality}, consider in $\mathcal{P}(\Delta_h)\times\mathcal{P}(\Lambda_h)\times\mathcal{P}(\Lambda_h)$ the germ $S_h$ at $(P_h,L_h,M_h)$ of the variety of polynomials $(P,L,M)$ such that $P$ has its singular points in a neighboorhood of the double points of $P_h$ and such that $L$ and $M$ vanish at these singular points. Define $\partial\Delta_+^h$ and $\partial\Lambda_+^h$ as follows:
\begin{itemize}
\item $\partial\Delta_+^1=\emptyset$ and $\partial\Lambda_+^1=\emptyset$,
\item for $2\leq h\leq l$, $\partial\Delta_+^h=[(0,4h-5)-(4k,4h-5)]$ and $\partial\Lambda_+^h=[(0,h-1)-(k,h-1)]$.
\end{itemize}
One has $\#(\partial\Delta_+^h\cap\Z^2)-1=4k$ and $\#(\partial\Lambda_+^h\cap\Z^2)-1=k$, and it follows from Theorem \ref{theoremfinal} applied to $(P_h,L_h,M_h)$ that $S_h$ is the transversal intersection of smooth hypersurfaces
\begin{eqnarray}
\label{4}
\left\lbrace\varphi_r^{(h)}=0\right\rbrace, \:\:\:\:  r=1,...,d_h,
\end{eqnarray}
$$
d_h=\mbox{codim} S_h,
$$
where
$$
\begin{array}{cccc}
\varphi_r^{(h)} : & \mathcal{P}(\Delta)\times\mathcal{P}(\Lambda)\times\mathcal{P}(\Lambda) & \longrightarrow & \C \\
& \left(A'_{i,j},B'_{u,v},C'_{u,v}\right) & \longmapsto & \varphi_r^{(h)}\left(A'_{i,j},B'_{u,v},C'_{u,v}\right).
\end{array}
$$
Moreover, there is a subset 
$$
\Xi_h\subset(\Z^2)^3
\cap(\Delta_h\setminus\partial\Delta_+^h
\times(\Lambda_h\setminus\partial\Lambda_+^h)^2),
$$
such that $\mbox{card}(\Xi_h)=d_h$, and 
$$
\det\left(\frac{\partial\varphi_r^{(h)}}{\partial A'_{i,j},\partial B'_{u,v},\partial C'_{u',v'}}\right)_{
\begin{tiny}
\begin{array}{l}
r=1,...,d_h ; \\
\lbrace(i,j),(u,v),(u',v')\rbrace\in\Xi_h
\end{array}
\end{tiny}
}\neq 0
$$
at the point
$$
(A'_{i,j},B'_{u,v},C'_{u',v'})= \left\{
\begin{array}{ll}
(a_{i,j},b_{u,v},c_{u',v'}) &, \lbrace(i,j),(u,v),(u',v')\rbrace\in\Xi_h \\
0 &, \mbox{ otherwise }
\end{array}  \right.
$$
%$$
%\nu_M(i,j)=\nu_L(i,j)=hj-\frac{h(h-1)}{2}, \mbox{ if }(i,j)\in\Lambda_h,
%$$
To find out $A_{i,j}(t), B_{u,v}(t)$ and $C_{u',v'}(t)$ we plug
$$
A'_{i,j}=A_{i,j}t^{\nu_P^h(j)}, \,\, B'_{u,v}=B_{u,v}t^{\nu_L^h(v)}  \,\,\mbox{and}\,\,
C'_{u',v'}=C_{u',v'}t^{\nu_M^h(v')}.
$$
in (\ref{4}) for any $h=1,...,l$.
\begin{lemma}
\label{lemmaShustin}
One has
$$
\det\left(\frac{\partial\varphi_r^{(h)}}{\partial A_{i,j},\partial B_{u,v},\partial C_{u',v'}}\right)_{
\begin{tiny}
\begin{array}{l}
r=1,...,d_h ; \\
\lbrace(i,j),(u,v),(u',v')\rbrace\in\bigcup_{h}\Xi_h
\end{array}
\end{tiny}
}\neq 0
$$
\end{lemma}
By means of the implicit function theorem, we derive the existence of the desired functions $A_{i,j}(t), B_{i,j}(t)$ and $C_{i,j}(t)$.
\qed

\textit{Proof of Lemma \ref{lemmaShustin}.}
The sets $\Xi_h$ are disjoint by construction and the matrix 
$$
\left(\frac{\partial\varphi_r^{(h)}}{\partial A_{i,j},\partial B_{u,v},\partial C_{u',v'}}\right)_{
\begin{tiny}
\begin{array}{l}
r=1,...,d_h ; \\
\lbrace(i,j),(u,v),(u',v')\rbrace\in\bigcup_{h}\Xi_h
\end{array}
\end{tiny}
}
$$
takes a block-triangular form as $t=0$ with the nondegenerated blocks
$$
\left(\frac{\partial\varphi_r^{(h)}}{\partial A'_{i,j},\partial B'_{u,v},\partial C'_{u',v'}}\right)_{
\begin{tiny}
\begin{array}{l}
r=1,...,d_h ; \\
\lbrace(i,j),(u,v),(u',v')\rbrace\in\Xi_h
\end{array}
\end{tiny}
},
$$
$h=1,...,l$, on the diagonal.
\qed

\section{Counterexample to Viro's conjecture in tridegree $(4,4,2)$}
\label{construction2}
To prove Theorem \ref{theorembis}, we present a construction of a curve of bidegree $(8,8)$ with $4$ double points lying on the intersection of two curves of bidegree $(2,2)$. Consider the curve $\lbrace C^1_3=0\rbrace \cup\lbrace C^2_3=0\rbrace $ (see Section \ref{constructionasympt}) with Newton polytope $Conv\left(\left(0,0\right),\left(6,0\right),\left(0,3\right),\left(6,1\right)\right)$. Perturb the curve $\lbrace C^1_3=0\rbrace \cup\lbrace C^2_3=0\rbrace$ keeping $4$ double points, as depicted in Figure \ref{courbe4_1}. Complete the rectangle $Conv\left(\left(0,0\right),\left(8,0\right),\left(0,3\right),\left(8,3\right)\right)$ with other charts of polynomials, as depicted in Figure \ref{courbe4_1}. By Shustin's patchworking theorem for curves with double points, there exists a polynomial $P$ of bidegree $(8,8)$ whose chart is depicted in Figure \ref{courbe4_1}.
\begin{figure}[h!]
\centerline{
\includegraphics[width=9cm,height=4cm]{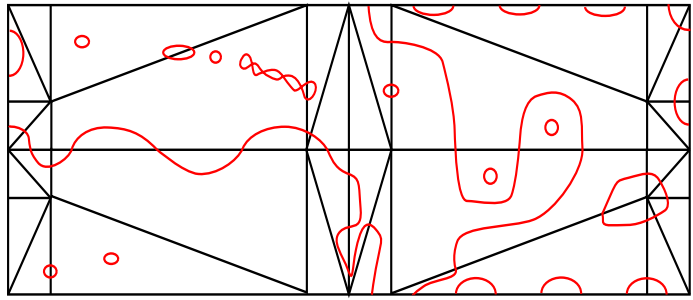}}
\setlength\abovecaptionskip{0cm}
\caption{The chart of $P$}
\label{courbe4_1}
\end{figure}
As in Lemma \ref{generalpositionlemma}, one can assume that the four double points of the curve $\lbrace P=0 \rbrace$ are the intersection points of two distinct irreducible nonsingular curves of bidegree $(2,2)$. Denote by $L(x,y)$ and $M(x,y)$ the polynomials defining the two curves of bidegree $(2,2)$ passing through the four double points of $\lbrace P=0 \rbrace$. % the equations of two curves of degree $(2,2)$ passing through the four double points.
 Put
\begin{itemize}
\item $P_1(x,y)=y^3P(x,\frac{1}{y})$,
\item $P_2(x,y)=y^3P(x,y)$,
\item $L_1(x,y)=yL(x,\frac{1}{y})$,
\item $L_2(x,y)=yL(x,y)$,
\item $M_1(x,y)=yM(x,\frac{1}{y})$,
\item $M_2(x,y)=yM(x,y)$.
\end{itemize}
Consider a Harnack curve $\left\lbrace P_3^0(x,y)=0\right\rbrace$ of bidegree $(8,2)$ in $(\RP^1)^2$ (see \cite{Mik00} for the definition of a Harnack curve). The chart of $P_3^0$ is depicted in Figure \ref{harnack}. Since $\left\lbrace P_3^0(x,y)=0\right\rbrace$ is a Harnack curve, one can assume that the restriction of $P_3^0$ to the edge $\left[\left(0,0\right),\left(8,0\right)\right]$ is equal to the restriction of $P$ to the edge $\left[\left(0,3\right),\left(8,3\right)\right]$ (see for example \cite{KenOkoun}). Put $P_3=y^6P_3^0$.
\begin{figure}[h!]
\centerline{
\includegraphics[width=9cm,height=4cm]{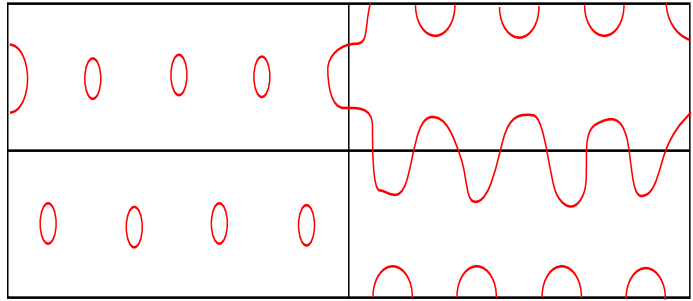}}
\setlength\abovecaptionskip{0cm}
\caption{The chart of $P_3^0$}
\label{harnack}
\end{figure}
\begin{theorem}
\label{theorem88}
There exist three real algebraic curves $\lbrace \tilde{P}=0\rbrace$, $\lbrace \tilde{L}=0 \rbrace$ and $\lbrace \tilde{M}=0\rbrace $ in $(\CP^1)^2$ such that:
\begin{itemize}
\item The polynomial $\tilde{P}$ is of bidegree $(8,8)$.
\item The polynomials $\tilde{L}$ and $\tilde{M}$ are of bidegree $(2,2)$.
\item The chart of $\tilde{P}$ is the result of the gluing of the charts of $P_1$, $P_2$ and $P_3$.
\item The $8$ double points of $\left(\tilde{P}=0\right)$ are on the intersection of the two curves $\left(\tilde{L}=0\right)$ and $\left(\tilde{M}=0\right)$.
\end{itemize}
\end{theorem}
\begin{proof}
The proof follows the same lines as the proof of Theorem \ref{theoremprincipal}.
\end{proof}
The chart of $\tilde{P}$ is depicted in Figure \ref{courbe4_2}.
\begin{figure}[h!]
\centerline{
\includegraphics[width=9cm,height=8cm]{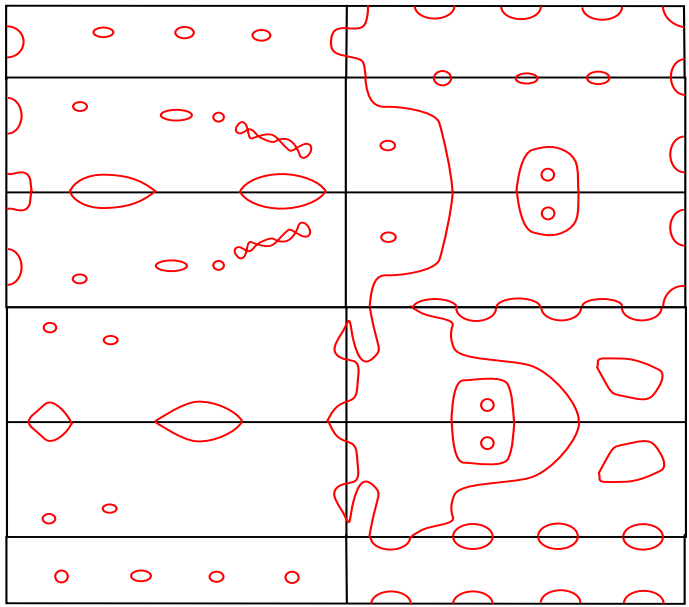}}
\setlength\abovecaptionskip{0cm}
\caption{The chart of $\tilde{P}$}
\label{courbe4_2}
\end{figure}
Denote by $\tilde{C}$ the strict transform of $\lbrace \tilde{P}=0\rbrace$ under the blow up of $(\CP^1)^2$ at the $8$ double points of $\lbrace \tilde{P}=0\rbrace $. Assume that the sign of $\tilde{P}$ in any empty oval is positive and consider the real algebraic surface $Z$ of tridegree $(4,4,2)$ in $(\CP^1)^3$ defined by
$$
Z=\lbrace F^2+\varepsilon G=0\rbrace,
$$
where $F=v_3\tilde{L}+u_3\tilde{M}$, the number $\varepsilon$ is some small positive parameter and $G$ is a polynomial of tridegree $(4,4,2)$ such that
$$
\tilde{C}=\lbrace F=0 \rbrace \cap \lbrace G=0 \rbrace.
$$
It follows from Section \ref{doublecovering} that $\R Z$ is homeomorphic to the disjoint union of two copies of $Y_-$ attached to each other by the identity map of $\tilde{C}$, where $Y_-$ is the part of $(\RP^1)^2$ blown up at the $8$ double points of $\lbrace \tilde{P}=0 \rbrace$ projecting to $\lbrace \tilde{P}\leq 0 \rbrace$. One can see from Figure \ref{courbe4_2} that $b_0(\R Z)=6$  and that $\R Z$ contains three spheres and two components of genus two. Moreover, one has
$$
\begin{array}{ccc}
\chi(\R Z) & = & 2(-8-(34-4))\\
& = & -76.
\end{array}
$$
Then $b_1(\R Z)=88$ and
$$
\R Z\simeq 2S_2\sqcup 3S\sqcup S_{40},
$$
which proves Theorem \ref{theorembis}.
%Then it follows from Section \ref{doublecovering} that there exists a real algebraic surface $X$ of tridegree $(4,4,2)$ in $(\CP^1)^3$ satisfying $b_1(\R X)=88$ and
%$$
%\R X\simeq 2S_2\sqcup 3S\sqcup S_{40}.
%$$	

\section{Transversality theorems}
\label{Transversality}
In this section, we prove a transversality theorem needed in the proof of Theorem \ref{theoremprincipal} and Theorem \ref{theorem88}.
\subsection{Notations}
\begin{itemize}
\item For a polytope $\Delta$, denote by $\vert\mathcal{L}_\Delta\vert$ the linear system on $Tor(\Delta)$ of curves of Newton polytope $\Delta$.
\item Let $F\in\mathcal{P}(\Delta)$, and let $\partial\Delta_+\subset\partial\Delta$ be a subset of the set of edges of $\Delta$. Introduce the space of polynomials
$$
\mathcal{P}(\Delta,\partial\Delta_+,F)=\lbrace G\in\mathcal{P}(\Delta) \mbox{ }\vert\mbox{ } G^{\sigma}=F^{\sigma}, \sigma\in\partial\Delta_+\rbrace.
$$
\item For a polytope $\Delta$, denote by $A(\Delta)$ the euclidean area of $\Delta$, by $b(\Delta)$ the number of integral points of the boundary of $\Delta$, by $i(\Delta)$ the number of integral points in the interior of $\Delta$ and put $\vert\Delta\vert=i(\Delta)+b(\Delta)$ the number of integral points in $\Delta$. 
%\item For any curve $C\in Tor(\Delta)$, denote by $K_C$ its canonical class. Remind that if $C\in\vert\mathcal{L}_\Delta\vert$, then $K_C=-b(\Delta)$.
%\item Denote by $\Sigma(d)$ the space of polynomials of degree $\leq d$ in two variables.
\end{itemize}
\subsection{Brusotti theorem}
Let $\Delta$ be a polytope. Denote by $t=(x,y)$ the coordinates on $(\C^*)^2$. In $\mathcal{P}(\Delta)\times(\C^*)^2$, consider the algebraic variety defined by
$$
B=\left\lbrace \begin{array}{c@{=\mbox{}}c}
P(t) & 0,\\
\partial P/\partial x(t) & 0,\\
\partial P/\partial y(t) & 0.\\
\end{array}\right. 
$$
Let $(P_0,t_0)\in B$, and assume that $t_0=(x_0,y_0)\in(\C^*)^2$ is an ordinary quadratic point of $\lbrace P_0=0\rbrace$.
\begin{lemma} $($see, for example, \cite{Risler} $)$
\\There exists a neighborhood $U$ of $(P_0,t_0)$ in $\mathcal{P}(\Delta)\times(\C^*)^2$ such that:
\begin{itemize}
\item $B\cap U$ is smooth of codimension $1$.
\item If $\pi:\mathcal{P}(\Delta)\times(\C^*)^2\rightarrow\mathcal{P}(\Delta)$ denotes the first projection, then $B'=\pi(B\cap U)$ is smooth of codimension $1$.
\end{itemize}
\end{lemma}
Assume now that the curve $\lbrace P_0=0\rbrace$ has $N$ non-degenerated double points $t_h$, \linebreak $h=1,...,N$ and no other singular points. Applying the above lemma to each $t_h$, we obtain $N$ non-singular analytic submanifolds of $\mathcal{P}(\Delta)$, say $B'_1,...,B'_N$ passing through $P_0$.
\begin{theorem}$($Brusotti, see, for example, \cite{Risler} $)$
\label{Brusottitheorem}
\\There exists a neighborhood $U'$ of $P_0$ in $\mathcal{P}(\Delta)$ such that the intersection 
$$
B'_1\cap\cdots\cap B'_N\cap U'
$$
is transversal in $\mathcal{P}(\Delta)$.
\end{theorem} 
The key point of the proof of Brusotti theorem is the following corollary of Riemann-Roch theorem.
\begin{lemma} $($see, for example, \cite{Risler} $)$
\label{Brusottilemma}
\\Let $\Delta$ be a polytope and let $C\in\vert\mathcal{L}_\Delta\vert$. Suppose that $C$ has $k$ non-degenerated double points $x_1,...,x_k$ in $(\C^*)^2$ and no other singular points.
Then, the linear subsystem of $\vert\mathcal{L}_\Delta\vert$ consisting of curves passing through $x_1,...x_k$ is of codimension $k$. 
\end{lemma} 

\subsection{Linear system of curves with prescribed quadratic points}
Let $\Delta$ be a polytope and let $C\in\vert\mathcal{L}_\Delta\vert$. Suppose that $C$ is irreducible with $k$ ordinary quadratic points $x_1,...,x_k$ in $(\C^*)^2$ and no other singular points. Fix also $m$ marked points $p_1,...p_m$ on $C\setminus\lbrace x_1,...,x_k\rbrace$.
\begin{lemma}
\label{theoremRR}
Suppose that 
$$
2k+m<b(\Delta).
$$
Then, the sublinear system of $\vert\mathcal{L}_\Delta\vert$ consisting of curves having singularities at $x_1,...,x_k$ and passing through $p_1,...,p_m$ is of codimension $3k+m$.
\end{lemma}
%\begin{remark}
%This means that $N$ double points and $k$ marked points of a curve in $Tor(\Delta)$ are in general position for the linear system associated to $\Delta$.
%\end{remark}
\begin{proof}
One has the following exact sheaf sequence:
$$
0\rightarrow\mathcal{O}_{Tor(\Delta)}\rightarrow\mathcal{O}_{Tor
(\Delta)}(C)\rightarrow\mathcal{O}_{C}(C)\rightarrow 0,
$$
where $\mathcal{O}_{Tor(\Delta)}$ is the sheaf of holomorphic functions on $Tor(\Delta)$, the sheaf $\mathcal{O}_{Tor(\Delta)}(C)$ is the invertible sheaf associated to the divisor $C$, and 
$\mathcal{O}_{C}(C)$ denotes the restriction of $\mathcal{O}_{Tor(\Delta)}(C)$ to $C$. As $H^1(Tor(\Delta),\mathcal{O}_{Tor(\Delta)})=0$, the long exact sequence in cohomology associated to the above exact sequence splits. The first part of the long exact sequence is the following:
$$
0\rightarrow\C\rightarrow H^0(Tor(\Delta),\mathcal{O}_{Tor
(\Delta)}(C))\overset{r}{\rightarrow}H^0(C,\mathcal{O}_{C}(C))\rightarrow 0.
$$ 
Denote by $E(x_1,...,x_k,p_1,...,p_m)$ the subspace of $H^0(Tor(\Delta),\mathcal{O}_{Tor
(\Delta)}(C))$ consisting of sections vanishing at least at order $2$ at $x_1,...,x_k$ and passing through $p_1,...,p_m$. Put $F=r(E(x_1,...,x_k,p_1,...,p_m))$. One has the following exact sequence:
$$
0\rightarrow\C\rightarrow E(x_1,...,x_k,p_1,...,p_m)\overset{r}{\rightarrow}F\rightarrow 0.
$$

Fix a generic section $s\in E(x_1,...,x_k,p_1,...,p_m)$ with divisor $D$. Then, $D\cap C$ consists of a finite number of points ($C$ is irreducible), and it defines a divisor on $C$ and also on $\tilde{C}$, the normalization of $C$. Denote by $(\tilde{x_i},\tilde{x_i}')$ the inverse images of $x_i$ by the normalization map. Denote by $\tilde{p_i}$ the inverse image of $p_i$ by the normalization map. Define on $\tilde{C}$ the divisor 
$$D':=D\cap C-E,$$
where
$$
E=\left\lbrace 2\sum_{i=1}^k(\tilde{x_i}+\tilde{x_i}')+\sum_{i=1}^m\tilde{p_i}\right\rbrace.$$ 
By definition of $r$, the set $F$ is the subspace of $H^0(C,\mathcal{O}_C(C\cap D))$ of sections vanishing at least at order $2$ at the points $x_1,...,x_k$ and at least at order $1$ at the points $p_1,...,p_m$. Considering the normalization map, one gets the following injective map:
$$
0\rightarrow H^0(C,\mathcal{O}_C(D\cap C))\overset{\iota}{\rightarrow} H^0(\tilde{C},\mathcal{O}_{\tilde{C}}(D\cap C)).
$$
Thus, $\dim(F)\leq\dim(\iota(F))$. The linear system $\iota(F)$ is the linear system of sections of $\mathcal{O}_{\tilde{C}}(D\cap C)$ vanishing at least at order $2$ at the points $(\tilde{x_i},\tilde{x_i}')$ and at least at order $1$ at the points $p_i'$. Consider the following exact sheaf sequence:
$$
0\rightarrow\mathcal{O}_{\tilde{C}}(D')\xrightarrow{\otimes s}\mathcal{O}_{\tilde{C}}(D\cap C)\xrightarrow{r} \mathcal{O}_{E}(D\cap C\mid_E)\rightarrow 0, 
$$
where $r$ is the restriction map. Passing to the associated long exact sequence, one sees that $\iota(F)$ is identified with $H^0(\tilde{C},D')$.
%\bigotimes_{1\leq i\leq m}\mathcal{O}_{\tilde{p_i}}\bigotimes_{1\leq i\leq n}\mathcal{O}_{\tilde{x_i}}\bigotimes_{1\leq i\leq n}\mathcal{O}_{\tilde{x_i}'}\rightarrow 0.
Let us compute $h^0(\tilde{C},D')$. The divisor $D'$ is of degree $A(2\Delta)-2A(\Delta)-4k-m$. Then, one has
$$
\begin{array}{c@{=\mbox{ }}c}
\deg(K_{\tilde{C}}-D')\mbox{ } & -\deg(D')-2+2g(\tilde{C})\\
& -A(2\Delta)+2A(\Delta)+4k+m-2+2(i(\Delta)-k)\\
& 2k+m-b(\Delta).\\
\end{array}
$$
By hypothesis, $2k+m-b(\Delta)<0$. So $h^0(\tilde{C},K_{\tilde{C}}-D')=0$, and by Riemann-Roch formula, one has
$$
\begin{array}{c@{=\mbox{ }}c}
h^0(\tilde{C},D')\mbox{ } & \deg(D')+1-g(\tilde{C})\\
 & A(2\Delta)-2A(\Delta)-4k-m+1-(i(\Delta)-k)\\
 & i(\Delta)+b(\Delta)-1-3k-m\\
 & \dim(\mathcal{L}(\Delta))-3k-m.
\end{array}
$$
So one gets
$$
\dim(F)\leq\dim(\mathcal{L}(\Delta))-3k-m.
$$
On the other hand
$$
\dim(F)=\dim(\PP(E(x_1,...,x_k,p_1,...,p_m))),
$$ 
and
$$
\dim(\PP(E(x_1,...,x_k,p_1,...,p_m)))\geq \dim(\mathcal{L}(\Delta))-3k-m.
$$
So finally
$$
\dim(\PP(E(x_1,...,x_k,p_1,...,p_m)))=\dim(\mathcal{L}(\Delta))-3k-m.
$$
\end{proof}
\subsection{A transversality theorem}
Fix three polytopes $\Delta,\Delta'$ and $\Delta''$. In $\mathcal{P}(\Delta)\times (\C^*)^2\times\mathcal{P}(\Delta')\times\mathcal{P}(\Delta'')$, consider the algebraic variety defined by
$$
S=\left\lbrace \begin{array}{c@{=\mbox{}}c}
P(t) & 0,\\
\partial P/\partial x(t) & 0,\\
\partial P/\partial y(t) & 0,\\
Q(t) & 0,\\
R(t) & 0.\\
\end{array}\right. 
$$
Let $(P_0,Q_0,R_0)\in \mathcal{P}(\Delta)\times\mathcal{P}(\Delta')\times\mathcal{P}(\Delta'')$, and assume that $t_0=(x_0,y_0)$ is an ordinary quadratic point of $\lbrace P_0=0\rbrace $ such that $\lbrace Q_0=0\rbrace$ intersects $\lbrace R_0=0\rbrace $ transversely at $t_0$. Then, in particular,  $(P_0,t_0,Q_0,R_0)\in S$. %Denotes the coordinates of the tangent space to $\mathcal{P}(\Delta)\times (\C^*)^2\times\mathcal{P}(\Delta')\times\mathcal{P}(\Delta'')$ at $(P_0,,t_0,Q_0,R_0)$ by $(A_{i,j},t=(u,v),B_{k,l},C_{m,n})$.
\begin{lemma}
There exists a neighborhood $U$ of $(P_0,t_0,Q_0,R_0)$ in $\mathcal{P}(\Delta)\times (\C^*)^2\times\mathcal{P}(\Delta')\times\mathcal{P}(\Delta'')$ such that:
\begin{itemize}
\item $S\cap U$ is smooth of codimension $5$.
\item If 
$$\pi:\mathcal{P}(\Delta)\times (\C^*)^2\times\mathcal{P}(\Delta')\times\mathcal{P}(\Delta'')\longrightarrow \mathcal{P}(\Delta)\times\mathcal{P}(\Delta')\times\mathcal{P}(\Delta'')$$ is the projection forgetting the second factor, then $S'=\pi(S\cap U)$ is smooth of codimension $3$. Moreover, the tangent space to $S'$ at $\pi(P_0,t_0,Q_0,R_0)$ is given by the equations 
$$
\left\lbrace\begin{array}{c@{=\mbox{}}c}
\sum A_{i,j}x_0^i y_0^j & 0,\\
-\begin{pmatrix}
d_{t_0}Q_0\\
d_{t_0}Q_0\\
\end{pmatrix}
(Hess_{t_0}P_0)^{-1}
\begin{pmatrix}
\sum iA_{i,j}x_0^{i-1}y_0^{j}\\
\sum jA_{i,j}x_0^{i}y_0^{j-1}\\
\end{pmatrix}
+
\begin{pmatrix}
\sum B_{k,l}x_0^k y_0^l\\
\sum C_{m,n}x_0^m y_0^n\\
\end{pmatrix}
& 0,
\end{array}\right.
$$
\end{itemize}
where $\left(A_{i,j}\right)_{i,j\in\Delta}$ are coordinates in the tangent space of $\mathcal{P}(\Delta)$, the coordinates $\left((B_{k,l}\right)_{k,l\in\Delta'}$ are coordinates in the tangent space of  $\mathcal{P}(\Delta')$, and $\left(C_{m,n}\right)_{m,n\in\Delta''}$ are coordinates in the tangent space of  $\mathcal{P}(\Delta'')$.

\end{lemma} 
\begin{proof}
The first point follows from the implicit function theorem. Introduce the map 
$$
F=(F_1,F_2,F_3,F_4,F_5):\mathcal{P}(\Delta)\times (\C^*)^2\times\mathcal{P}(\Delta')\times\mathcal{P}(\Delta'')\longrightarrow\C^5,
$$
with 
\begin{itemize}
\item $F_1(P,t,Q,R)=P(t)$, 
\item $F_2(P,t,Q,R)=\partial P/\partial x(t)$, 
\item $F_3(P,t,Q,R)=\partial P/\partial y(t)$, 
\item $F_4(P,t,Q,R)=Q(t)$, 
\item $F_5(P,t,Q,R)=R(t)$.
\end{itemize}
Denote by $J_F(P_0,t_0,Q_0,R_0)$ the Jacobian matrix of $F$ at $(P_0,t_0,Q_0,R_0)$. By an easy computation, one has
$$
J_F(P_0,t_0,Q_0,R_0)=
\begin{pmatrix}
 x_0^iy_0^j & \begin{matrix}
 0 & 0
 \end{matrix} & \begin{matrix}
 0 & \quad\:\: 0
 \end{matrix}\\
 \begin{matrix}
 ix_0^{i-1}y_0^j \\ 
 jx_0^iy_0^{j-1} \\
 \end{matrix} &
 \begin{matrix}
  \mathlarger{Hess_{t_0} P_0}
 \end{matrix} &
 \begin{matrix}
 0 & \quad\:\: 0 \\ 
 0 & \quad\:\: 0 \\
 \end{matrix} \\
 0 & \begin{matrix}
 d_{t_0}Q_0
 \end{matrix} & \begin{matrix}
 x_0^ky_0^l & 0
 \end{matrix}\\
  0 & \begin{matrix}
 d_{t_0}R_0
 \end{matrix} & \begin{matrix}
 \:\:\:\: 0 & x_0^my_0^n
 \end{matrix}\\
\end{pmatrix}_{\left[ 5,\vert\Delta\vert+2+\vert\Delta'\vert+\vert\Delta''\vert \right]}.
$$
As $t_0$ is an ordinary quadratic point of $P_0$, the matrix $Hess_{t_0} P_0$ is invertible. 
It follows that $J_F(P_0,t_0,Q_0,R_0)$ is of rank $5$. In fact, the submatrix of $J_F(P_0,t_0,Q_0,R_0)$, where $i=j=0$, $k=l=0$ and $m=n=0$, is as follows:
$$
\begin{pmatrix}
 1 & \begin{matrix}
 0 & 0
 \end{matrix} & \begin{matrix}
 0 & 0
 \end{matrix}\\
 \begin{matrix}
 0 \\ 
 0 \\
 \end{matrix} &
 \begin{matrix}
  \mathlarger{Hess_{t_0} P_0}
 \end{matrix} &
 \begin{matrix}
 0 & 0 \\ 
 0 & 0 \\
 \end{matrix} \\
 0 & \begin{matrix}
 d_{t_0}Q_0
 \end{matrix} & \begin{matrix}
 1 & 0
 \end{matrix}\\
  0 & \begin{matrix}
 d_{t_0}R_0
 \end{matrix} & \begin{matrix}
 0 & 1
 \end{matrix}\\
\end{pmatrix}.
$$
This last matrix is invertible. For the second point, since $Hess_{t_0} P_0$ is invertible, use the second and the third equations of the tangent space to $S$ at $(P_0,t_0,Q_0,R_0)$ to write $t$ as a function of $P$ over a small neighborhood of $(P_0,Q_0,R_0)$ in $\mathcal{P}(\Delta)\times\mathcal{P}(\Delta')\times\mathcal{P}(\Delta'')$. This proves the lemma.
\end{proof}
Assume now that the curve $\lbrace P_0=0\rbrace$ is irreducible and has $N$ non-degenerated double points $t_h$, $h=1,...,N$ such that $\lbrace Q_0=0\rbrace$ intersects $\lbrace R_0=0\rbrace $ transversely at each $t_h$. Assume also that $\lbrace P_0=0\rbrace$ has no further singular points. Applying the above lemma to each $t_h$, we obtain $N$ nonsingular analytic manifolds of $\mathcal{P}(\Delta)\times\mathcal{P}(\Delta')\times\mathcal{P}(\Delta'')$ passing through $(P_0,Q_0,R_0)$. Denote these $N$ nonsingular analytic manifolds by $S'_1,\cdots ,S'_N$, 
\begin{theorem}
\label{theoremtrans}
If $2N<b(\Delta)$, then
there exists a neighborhood $W$ of $(P_0,Q_0,R_0)$ in $\mathcal{P}(\Delta)\times\mathcal{P}(\Delta')\times\mathcal{P}(\Delta'')$ such that the intersection 
$$
S'_1\cap\cdots\cap S'_N\cap W
$$ 
is transversal in $\mathcal{P}(\Delta)\times\mathcal{P}(\Delta')\times\mathcal{P}(\Delta'')$.
\end{theorem}
\begin{proof}
By the implicit function theorem, it is sufficient to show that the tangent spaces to the manifolds $S'_h$ at $(P_0,Q_0,R_0)$ intersect transversely. This is equivalent to the fact that the matrix
$$
\scriptsize{M=\begin{pmatrix}
x_1^iy_1^j & 0 & 0 \\
\begin{pmatrix}
d_{t_1}Q_0\\
d_{t_1}R_0\\
\end{pmatrix}
(Hess_{t_1}P_0)^{-1}
\begin{pmatrix}
ix_1^{i-1}y_1^j\\
jx_1^iy_1^{j-1}\\
\end{pmatrix}
&
\begin{pmatrix}
x_1^ky_1^l\\
0\\
\end{pmatrix} &
\begin{pmatrix}
0\\
x_1^my_1^n\\
\end{pmatrix}\\
\vdots & \vdots & \vdots\\
\vdots & \vdots & \vdots\\
x_N^iy_N^j & 0 & 0 \\
\begin{pmatrix}
d_{t_N}Q_0\\
d_{t_N}R_0\\
\end{pmatrix}
(Hess_{t_N}P_0)^{-1}
\begin{pmatrix}
ix_N^{i-1}y_N^j\\
jx_N^iy_N^{j-1}\\
\end{pmatrix}
&
\begin{pmatrix}
x_N^ky_N^l\\
0\\
\end{pmatrix} &
\begin{pmatrix}
0\\
x_N^my_N^n\\
\end{pmatrix}\\
\end{pmatrix}_{\left[3N,\vert\Delta\vert+\vert\Delta'\vert+\vert\Delta''\vert \right]}}$$
is of rank $3N$. As 
$\lbrace Q_0=0\rbrace$ and $\lbrace R_0=0\rbrace $ intersect transversely at each $t_h$,
the matrices 
$$
\begin{pmatrix}
d_{t_1}Q_0\\
d_{t_1}R_0\\
\end{pmatrix},
\cdots,
\begin{pmatrix}
d_{t_N}Q_0\\
d_{t_N}R_0\\
\end{pmatrix}
$$
are invertible.
Consider the following matrix:
$$ \scriptsize{N=
\begin{pmatrix}
1 & 
\begin{matrix}
0 & 0\\
\end{matrix} & 0 & \cdots & 0 & \begin{matrix}
0 & 0\\
\end{matrix} \\
\begin{matrix}
0\\
0\\
\end{matrix} & 
(Hess_{t_1}P_0)
\begin{pmatrix}
d_{t_1}Q_0\\
d_{t_1}R_0\\
\end{pmatrix}^{-1} & \begin{matrix}
0\\
0\\
\end{matrix} & \begin{matrix}
\cdots\\
\cdots\\
\end{matrix} & 
\begin{matrix}
0\\
0\\
\end{matrix} & 
\begin{matrix}
0 & 0\\
0 & 0\\
\end{matrix}\\
\vdots & \vdots & \ddots & \cdots & \vdots & \vdots \\
0 & 
\begin{matrix}
0 & 0\\
\end{matrix}
& 0 & \cdots & 1 & 
\begin{matrix}
0 & 0\\
\end{matrix}\\
\begin{matrix}
0\\
0\\
\end{matrix} & 
\begin{matrix}
0 & 0\\
0 & 0\\
\end{matrix} & 
\begin{matrix}
0\\
0\\
\end{matrix} & 
\begin{matrix}
\cdots\\
\cdots\\
\end{matrix} & 
\begin{matrix}
0\\
0\\
\end{matrix} &
(Hess_{t_N}P_0)
\begin{pmatrix}
d_{t_N}Q_0\\
d_{t_N}R_0\\
\end{pmatrix}^{-1}\\

\end{pmatrix}_{\left[3N,3N\right]}.}
$$
The product $NM$ is of the following form:
$$
\begin{pmatrix}
x_1^iy_1^j & 0 & 0 \\
\begin{pmatrix}
ix_1^{i-1}y_1^j\\
jx_1^iy_1^{j-1}\\
\end{pmatrix}
&
\begin{pmatrix}
*\cdots *\\
*\cdots *\\
\end{pmatrix} &
\begin{pmatrix}
*\cdots *\\
*\cdots *\\
\end{pmatrix}\\
\vdots & \vdots & \vdots\\
\vdots & \vdots & \vdots\\
x_N^iy_N^j & 0 & 0 \\
\begin{pmatrix}
ix_N^{i-1}y_N^j\\
jx_N^iy_N^{j-1}\\
\end{pmatrix}
&
\begin{pmatrix}
*\cdots *\\
*\cdots *\\
\end{pmatrix} &
\begin{pmatrix}
*\cdots *\\
*\cdots *\\
\end{pmatrix}\\
\end{pmatrix}_{\left[3N,\vert\Delta\vert+\vert\Delta'\vert+\vert\Delta''\vert\right]}.
$$
It is then sufficient to show that the matrix
$$
\begin{pmatrix}
x_1^iy_1^j\\
\begin{pmatrix}
ix_1^{i-1}y_1^j\\
jx_1^iy_1^{j-1}\\
\end{pmatrix}\\
\vdots\\
x_N^iy_N^j\\
\begin{pmatrix}
ix_N^{i-1}y_N^j\\
jx_N^iy_N^{j-1}\\
\end{pmatrix}\\
\end{pmatrix}_{\left[3N,\vert\Delta\vert\right]}.
$$
is of rank $3N$.
It is equivalent to show that the linear space of curves in $\vert\mathcal{L}_\Delta\vert$ with singularities at all the quadratic points of $({P_0=0})$ is of codimension $3N$. This follows from Lemma \ref{theoremRR}.
\end{proof}
Let $\partial\Delta_+\subset\partial\Delta$ be a subset of the set of edges of $\Delta$. Let $\partial\Delta'_+\subset\partial\Delta'$ (resp., \linebreak $\partial\Delta''_+\subset\partial\Delta''$) be a subset of the set of edges of $\Delta'$ (resp., of $\Delta''$).
Put
$$
m=\sum_{F\in\partial\Delta_+}\left(\#(F\cap\Z^2)-1\right),
$$
$$
m'=\sum_{F'\in\partial\Delta'_+}\left(\#(F'\cap\Z^2)-1
\right),
$$
$$
m''=\sum_{F''\in\partial\Delta''_+}\left(\#(F''\cap\Z^2)-1
\right).
$$
%Let $\partial\Delta_+\subset\partial\Delta$ be a union of $k$ facets of $\Delta$. Let $\partial\Delta'_+\subset\partial\Delta'$ (resp. $\partial\Delta''_+\subset\partial\Delta''$) be a union of $k$  facets of $\Delta'$ (resp. $k'$ facets of $\Delta''$). Let us denote by $l$ the number of interior points of $\partial\Delta_+$, and put $m=k+l$. Let us denote by $l'$ (resp. $l''$) the number of interior points of $\partial\Delta'_+$ (resp. $\partial\Delta''_+$). Put $m'=k'+l'$ and $m''=k''+l''$.
\begin{theorem}
\label{theoremfinal}

Suppose that 
$$
\left\lbrace\begin{array}{c@{<\mbox{ }}c}
2N+m & b(\Delta),\\
m' & b(\Delta'),\\
m'' & b(\Delta'').\\
\end{array}\right.
$$ 
Then, there exists a neighborhood $V$ of $(P_0,t_0,Q_0,R_0)$ such that the intersection
$$
S'_1\cap ...\cap S'_N\cap\left(\mathcal{P}(\Delta,\partial\Delta_+,P_0)\times\mathcal{P}(\Delta',\partial\Delta'_+,Q_0)\times\mathcal{P}(\Delta'',\partial\Delta''_+,R_0)\right)\cap V
$$
is transversal in $\mathcal{P}(\Delta)\times\mathcal{P}(\Delta')\times\mathcal{P}(\Delta'')$.
\end{theorem}
\begin{proof}
The proof follows the same lines as the proof of Theorem \ref{theoremtrans}. It reduces to the proof of the following facts.
\begin{itemize}
\item The linear space of curves in $\vert\mathcal{L}_\Delta\vert$ with singularities at all the quadratic points of $({P_0=0})$ and passing through the $m$ points of intersection of $({P_0=0})$ with $\partial\Delta_+$ is of codimension $3N+m$.
\item The linear space of curves in $\vert\mathcal{L}_{\Delta'}\vert$  passing through the $m'$ points of intersection of $({Q_0=0})$ with $\partial\Delta'_+$ is of codimension $m'$.
\item The linear space of curves in $\vert\mathcal{L}_{\Delta''}\vert$  passing through the $m''$ points of intersection of $({R_0=0})$ with $\partial\Delta''_+$ is of codimension $m''$.
\end{itemize} 
All this follows from Lemma \ref{theoremRR}.
\end{proof}

\bibliographystyle{alpha}
\bibliography{biblio}

\end{document}